\documentclass[a4paper]{paper}
\usepackage{amsmath,amsthm,amsfonts,mathrsfs}
\usepackage{amssymb}
\usepackage{tikz-cd}
\usepackage{xspace}
\usepackage{graphicx}
\usepackage{paralist}
\usepackage{hyperref} 

\usepackage{multirow}
\usepackage{diagbox}

\hypersetup{%
    pdfmenubar=true,       
    pdffitwindow=false,     
    pdfstartview={FitH},    
    pdftitle={A New Variational Model for Joint Image Reconstruction and Motion Estimation in Spatiotemporal Imaging},            
    colorlinks=true,       
    linkcolor=red,          
    citecolor=red,        
    filecolor=magenta,      
    urlcolor=cyan,           
    pdfborder = {0,0,0}
}
\usepackage{acro}
\usepackage{cleveref}
\usepackage{algorithm}
\usepackage[noend]{algpseudocode}
\makeatletter
\def\BState{\State\hskip-\ALG@thistlm}
\makeatother

\Crefname{equation}{}{}
\crefname{equation}{}{}

\usepackage{todonotes}

\usepackage{specialmacros}
\usepackage{paperacronyms}
\acuse{ODE}
\acuse{PDE}

\begin{document}

\title{A New Variational Model for Joint Image Reconstruction and Motion Estimation in Spatiotemporal Imaging}
\author{Chong Chen\thanks{LSEC, ICMSEC, Academy of Mathematics and Systems Science, Chinese Academy of Sciences, Beijing 100190, China (chench@lsec.cc.ac.cn).}, Barbara Gris\thanks{LJLL--Laboratoire Jacques-Louis Lions, Sorbonne Universit\'e, 75005 Paris, France (bgris.maths@gmail.com).} and Ozan \"Oktem\thanks{Department of Mathematics, KTH--Royal Institute of Technology, 10044 Stockholm, Sweden (ozan@kth.se).}}

\maketitle

\begin{abstract}
We propose a new variational model for joint image reconstruction and motion 
estimation in spatiotemporal imaging, which is investigated along a general 
framework that we present with shape theory. 
This model consists of two components, one for conducting modified static image reconstruction, 
and the other performs sequentially indirect image registration. 
For the latter, we generalize the large deformation diffeomorphic metric mapping 
framework into the sequentially indirect registration setting. 
The proposed model is compared theoretically against alternative approaches (optical flow based model and 
diffeomorphic motion models), and we demonstrate that the proposed model has desirable properties in terms of 
the optimal solution. The theoretical derivations and efficient algorithms are also presented for a time-discretized 
scenario of the proposed model, which show that the optimal solution of the time-discretized version is consistent with that of the 
time-continuous one, and most of the computational components is the easy-implemented linearized 
deformation. The complexity of the algorithm is analyzed as well. This work is concluded by some numerical 
examples in 2D space + time tomography with very sparse and/or highly noisy data.
\end{abstract}

\begin{keywords}
spatiotemporal imaging, image reconstruction, motion estimation, joint variational model, shape theory, large diffeomorphic deformations
\end{keywords}

\section{Introduction}
Image reconstruction is challenging in a spatiotemporal setting when the object being imaged undergoes a temporal evolution. 
This is the case in tomographic imaging of the heart or lung \cite{BuDiSch18,GiJiDaSc15} where it is important 
to estimate and compensate for the unknown motion of the organs. As an example, data in \ac{PET} cardiac 
imaging is acquired over a relatively long period of time (often in the range of minutes). The respiratory and 
cardiac motion cause a displacement of 20-40~mm to organs of interest \cite{ScLe00,WaViBe99}. 
Failing to correct for such motion leads to a degradation in image quality \cite{GiJiDaSc15}.

Data in spatiotemporal imaging is a time or quasi-time series and one obvious approach is to decompose 
it into sub-sets (gates) such that data within each gate is generated by the object in a fixed temporal 
state \cite{DaStJiScScSc09,LaDaBuScScSc06,Luci09}. Much work has been done along these lines in the 
context of \ac{PET} imaging of the heart or lung, where each gate is the phase of breathing and/or 
cardiac motion \cite{DaBuLaScSc07,GrGeMe14,BuDaWuSc09}. One also studies on how to optimize 
the gating in order to obtain the optimal image quality by reconstruction \cite{ElHaJoDe11,VaSeHuPa14}. 

Once data is gated, algorithms for spatiotemporal image reconstruction can be separated into two categories:  
first image reconstruction then motion estimation; joint image reconstruction and motion 
estimation \cite{BlNaRa12,GiJiDaSc15}. In the first category, one starts with applying static image 
reconstruction on data from each of the gates, resulting in a series of (low-resolution) images, 
then selects a reconstructed image as target and registers the other reconstructed images against 
this target, finally averages all the registered images to obtain the reconstructed image 
(see, for example, \cite{BaBr09,BaBr11,DaBuJiSc08,DaLaJiSc08,GiRuBu12,GiJiDaSc15}). 
The second category is establishing the joint tasks of image reconstruction and motion 
estimation into one model, then gaining the optimal solution to reconstruct the image in 
each gate. This is more complex and several approaches have been suggested for how to do 
this, such as \cite{BlNaRa12,ScMoFi09,TaKu08,GaCaSh11,LiZhZhGa15,BrPaOeKuKa12,
BrPaOeKa12,RiSaKnKa12,JiLoDoTi10,TaCaSr12,BhPaMi12,HiSzWaSaJo12,HiSzWaSaJo12,BrSaMaKa15,BuDiSch18}, and so forth. 

The approach taken in this paper belongs to the second category and the motion model makes use 
of diffeomorphic deformations. The latter are provided by the \ac{LDDMM} framework, which is a well-developed 
framework for diffeomorphic image registration (see, for instance, \cite{Tro98,DuGrMi98,MiTrYo02,BeMiTrYo05,GrMi07,Yo10,TrYo11,TrYo15,BrHo15,BrVi17}). 
Diffeomorphic deformations based on \ac{LDDMM} were used in \cite{HiSzWaSaJo12} for joint image reconstruction 
and motion estimation in 4D \ac{CT}, which is based on the growth model of \ac{LDDMM} \cite{GrMi07}.  
Later they were also used in \cite{ChOz18} for indirect image registration.  Nevertheless, 
this paper is dedicated to proposing a new joint variational model based on the principle of \ac{LDDMM}.

\paragraph{Contributions} The main contribution is a new variational model for joint image reconstruction and motion 
estimation in spatiotemporal imaging based on the \ac{LDDMM} framework, which is studied along a general 
framework of variational model that we present with deformable templates from shape theory. 
This model contains two components: one corresponding to modified static image reconstruction, 
and the other corresponding to sequentially indirect image registration. 
For the latter, we generalize the \ac{LDDMM} framework 
into the sequentially indirect registration setting. 

The mathematical properties of the proposed variational model is compared against the optical flow 
based model in \cite{BuDiSch18} and the diffeomorphic motion model in \cite{HiSzWaSaJo12}.
The comparison shows that the proposed model has some desirable properties in terms of the optimal solution, for example, 
guaranteeing elastically large diffeomorphic deformations, averagely distributed \wrt time $t$, and 
non-vanishing neither on the initial nor on the end time point, etc.
Moreover, a computationally efficient gradient-based iterative scheme is presented for the time-discretized version. 
More importantly, the optimal solution of the time-discretized problem is consistent with that of the time-continuous one.
Most of the computationally demanding parts relate to implementing the linearized deformations \cite{OkChDoRaBa16}.

\paragraph{Outline} A general variational model for joint image reconstruction and motion estimation is presented  
in \cref{sec:GeneralFramework_Imaging}. For self-contained, we revisit \ac{LDDMM} briefly 
in \cref{LDDMM}, then propose the new variational model in \cref{sec:LDDMM_based_model}. 
The \cref{sec:LDDMM_based_model_related} makes the mathematical 
comparison between the proposed model and the existing models. The \cref{sec:computed_method} gives the detailed 
numerical algorithms for solving the proposed model. The numerical experiments are 
performed in \cref{sec:numerical _experiments} to show the performance of the new model with 2D + time tomography. 
Finally, the \cref{sec:Conclusions} concludes the paper.

\section{A general variational model for joint image reconstruction and motion estimation}
\label{sec:GeneralFramework_Imaging}

The spatiotemporal (space $+$ time) image reconstruction is typically a spatiotemporal inverse problem.  
The aim is to estimate a spatially distributed quantity (image) that exhibits 
temporal variations from indirect time-dependent noisy observations (measured data). Hence, both the image and 
its motion are unknown.

\subsection{General spatiotemporal inverse problem}\label{subsec:InverProb}

Let $\signal \colon [t_0, t_1] \times \domain \to \Real^k$ denote the spatiotemporal image that we need to reconstruct. 
Here $k$ is the number of modalities (often $k=1$) and $\domain \subset \Real^n$ is spatial domain. 
Without loss of generality, the general (quasi-)time domain $[t_0, t_1]$ can be reparameterized onto $[0,1]$.

The spatiotemporal inverse problem is to reconstruct a spatiotemporal 
image $\signal(t,\Cdot) \in \RecSpace$ from 
measured data $\data(t,\Cdot) \in\DataSpace$ such that
\begin{equation}\label{eq:InvProb}  
     \data(t,\Cdot) = \ForwardOp\bigl(t, \signal(t,\Cdot)\bigr) + \noisedata(t,\Cdot) \quad \text{for $t \in [0, 1]$}, 
\end{equation}
where $\RecSpace$ (reconstruction space) is the vector space of all 
possible images on a fixed domain $\domain$, $\DataSpace$ (data space) is the  
vector space of all possible data, and $\noisedata(t,\Cdot) \in \DataSpace$ is the observation noise in data.  
Furthermore, $\ForwardOp(t, \Cdot) \colon \RecSpace \to \DataSpace$ is a time-dependent forward operator, 
for short denoted by $\ForwardOp_t$, that models how an image at time $t$ gives rise to data in 
absence of noise or measurement errors (\eg  a stack of Radon transforms with various geometric parameters of scanning 
for \ac{CT} and \ac{PET} and attenuated Radon transform instead for \ac{SPECT}, \etc) \cite{Na01}.  
Note that for spatiotemporal problems, the collection of geometric parameters of scanning is 
basically different in the temporal direction.  

A key step is to further specify the form of the spatiotemporal image $\signal(t,\Cdot)$ and here 
we will use the idea of deformable templates from shape theory.

\subsection{Spatiotemporal inverse problem based on shape theory}\label{sec:SpatialTemporShapeTheory}

Shape theory seeks to develop quantitative tools to study shapes and their variability, 
which can be pursued to work of D'Arcy Thompson \cite{Ar45}. Shapes of objects are  
considered as points in the shape space \cite{TrYo15}. The collections of deformable objects and 
deformations need to be defined. The former represent the objects whose shape 
we want to analyze, the latter are transformations that can act on the deformable objects. The underlying idea is 
that shapes are represented as a deformation of a template, so the template represents the ``shape invariant'' part of the 
object whereas the set of deformations model how various shapes arise. Shape similarity between two objects 
can then be quantified as the ``cost'' of deforming one object into the other by means of a minimal deformation in 
the set of deformations. For more details, the reader is referred to \cite{Yo10,GrMi07,MiTrYo15}. 

Based on the thought above, the spatial and temporal components of a spatiotemporal image can be separated as
\begin{equation}\label{eq:SeparateImage}
\signal(t,\Cdot) := \DefOpV(\diffeo_t, \template) \quad \text{for some $\diffeo_t \in \LieGroup$ and $\template \in \RecSpace$}.
\end{equation}
Here $\DefOpV \colon \LieGroup \times \RecSpace \to \RecSpace$ is a temporal evolution operator 
and $\LieGroup$ is the group of diffeomorphisms on $\domain$, \ie the set of invertible 
mappings that are continuously differentiable with a continuously differentiable inverse 
from $\domain$ to $\domain$. It is natural to require that $\DefOpV$ is a group action 
of $\LieGroup$ on $\RecSpace$, so we can write $\DefOpV(\diffeo_t, \template) := \diffeo_t . \template$ for simplicity. 
Furthermore, $\template \colon \domain \to \Real$ (template) is the time-independent spatial component and the deformation 
$\diffeo_t \colon \domain \to \domain$ governs the temporal evolution of the template. 

Therefore, the inverse problem in spatiotemporal imaging can now be written as
\begin{equation}\label{eq:InvProb_2}  
     \data(t,\Cdot) = \ForwardOp_t(\diffeo_t . \template) + \noisedata(t,\Cdot) \quad \text{for $t \in [0, 1]$}. 
\end{equation}
Notice that  $\signal(t,\Cdot)=\diffeo_t . \template$ is the spatiotemporal image at time $t$ generated 
from the template $\template$ and the 
deformation $\diffeo_t$. Hence, the above inverse problem calls for simultaneously recovering 
the time-independent template $\template$ and the time-dependent deformation $\diffeo_t$. 

\paragraph{Common group actions} A natural group action is the one that gives geometric deformations \cite{ChOz18}: 
\begin{equation}\label{eq:GeometricDeformation}
\diffeo_t . \template = \template \circ \diffeo_t^{-1},
\end{equation}
where ``$\circ$" denotes function composition. This group action yields a deformation that 
merely moves the position of the pixel/voxel but does not change its intensity.

An alternative group action is the one that corresponds to mass-preserving deformations \cite{ChOz18,Yo10}:
\begin{equation}\label{eq:MassPreservedDeform}
\diffeo_t . \template = \bigl\vert \Diff(\diffeo_t^{-1}) \bigr\vert \template \circ \diffeo_t^{-1},
\end{equation}
where $\bigl\vert \Diff(\phi) \bigr\vert$ denotes the determinant of the Jacobian of $\phi$. 
This group action adjusts the intensity values but preserves the total mass.

\subsection{A general framework of joint image reconstruction and motion estimation}\label{sec:VariatRegCT}

First of all, through simple analysis, it is not difficult to observe that the 
whole inverse problem \cref{eq:InvProb_2} can be divided into two small subproblems. 
With given evolution deformation $\diffeo_t$,  the original problem boils down to a  
modified static image reconstruction problem: one just needs to reconstruct 
template ``$\template$" from noisy measured data. Since the data sets are measured from 
deformed template by known deformations, we name it with ``modified". On the other hand, 
with given template $\template$, the remaining goal is to estimate the evolution 
parameter ``$\diffeo_t$" from noisy measured data. We call 
it sequentially indirect image registration, which is a generalization of indirect 
image registration \cite{ChOz18}. 

The inverse problem in \cref{eq:InvProb_2} is ill-posed due to a variety of reasons. A variational formulation 
offers a flexible framework for regularizing a wide range of inverse problems \cite{ScGrGrHaLe09}.
The idea is to add regularization functionals penalize a maximum likelihood solution and thereby act as stabilisers. 

In \cref{eq:InvProb_2} we seek to recover both the template $\template \in \RecSpace$ and the 
temporal deformation $\diffeo_t \in \LieGroup$ simultaneously from time-series 
data $\data(t,\Cdot) \in \DataSpace$. A general variational model for this inverse problem reads as 
\begin{equation}\label{eq:VarReg_2}
 \min_{\substack{\template \in \RecSpace \\ \diffeo_t \in \LieGroup}} \biggl\{
    \int_{0}^{1} \Bigl[\DataDisc\bigl( \ForwardOp_t(\diffeo_t . \template), \data(t,\Cdot) \bigr) \dint t + \mu_2\RegFunc_2(\diffeo_t) \Bigr]\dint t + \mu_1 \RegFunc_1(\template)
    \biggr\}, 
\end{equation}  
where $\mu_1, \mu_2$ are nonnegative regularization parameters. 

The above $\DataDisc : \DataSpace \times \DataSpace \rightarrow \Real_+$ is the data discrepancy 
functional quantifying the mismatch in data space $\DataSpace$ that is often designed as the form of 
$\LpSpace^p$-norm or \ac{KL} divergence. The above $\Real_+$ denotes the set  
of nonnegative real number. The selection of $\DataDisc$ depends generally on the distribution of 
the noise in data, \eg applying squared $\LpSpace^2$-norm  
\begin{equation}\label{eq:DataMatching_L2}
\DataDisc\bigl( \data_1, \data_2 \bigr) := \bigr\Vert \data_1 - \data_2 \bigr\Vert^2_2
\end{equation}
to Gaussian distribution and \ac{KL} divergence
\begin{equation}\label{eq:DataMatching_KL}
\DataDisc\bigl( \data_1, \data_2 \bigr) := \int_{\domain_{\DataSpace}} \data_1(y) - \data_2(y) \ln \bigl(\data_1(y) \bigr) \dint y
\end{equation}
to Poisson distribution and for $ \data_1, \data_2 \in \DataSpace$.

Moreover, the spatial regularization $\RegFunc_1 : \RecSpace \to \Real_+$ introduces well-posedness 
by encoding priori knowledge about the image $f$, which is frequently based on the form 
of $\LpSpace^p$-norm of its gradient magnitude or certain sparse representation. 
Typically, taking the squared $\LpSpace^2$-norm of the gradient magnitude 
\begin{equation}\label{eq:TikReg}
\RegFunc_1(f) := \Vert \nabla f \Vert^2_2
\end{equation}
is known to produce smooth images whereas selecting the $\LpSpace^1$-norm of the gradient 
magnitude, \ie \ac{TV} regularization
\begin{equation}\label{eq:TVReg}
\RegFunc_1(f) := \Vert \nabla f \Vert_1
\end{equation}
yields edge-preserving images \cite{RuOs92}.

Subsequently, the key problem is to describe how to generate the evolution deformation $\diffeo_t$ and to select 
shape regularization $\RegFunc_2 : \LieGroup \to \Real_+$ for the $\diffeo_t$. 
We will consider this problem within the \ac{LDDMM} framework.

\section{A new variational model for joint image reconstruction and motion estimation}\label{sec:LDDMM_recon_model}

This section introduces a new variational model of the framework \cref{eq:VarReg_2} based on \ac{LDDMM}. 
First we recall the basic principle of \ac{LDDMM} for self-contained, and then give the proposed model.

\subsection{The \ac{LDDMM} framework} \label{LDDMM}

The \ac{LDDMM} framework outlined here offers a generic way to generate a flow of 
diffeomorphisms through velocity field. In this framework, the linearized deformation  
is considered as infinitesimal deformation, and the displacement field is seen as an 
instantaneous velocity field. Under certain regularity, the composition of such 
small deformations in the limit generates a flow of diffeomorphisms given 
as the solution to an ordinary differential equation (ODE) \cite{Yo10}.

More precisely, given a velocity field $\velocityfield : [0,1] \times \domain \to \Real^n$,  
a flow $\diffeo_t$ is generated by the following ODE: 
\begin{equation}\label{eq:FlowEq}
 \begin{cases} 
    \partial_t \diffeo_t(x) = \velocityfield\bigl( t, \diffeo_t(x) \bigr) & \\[0.5em]
    \diffeo_0(x)=x &  
   \end{cases} 
   \quad\text{for $x\in \domain$ and $0 \leq t \leq 1$.}    
\end{equation}
Note that $\diffeo_0 = \Id$, \ie the flow starts at the identity deformation (mapping). If the velocity field $\velocityfield$ is 
sufficiently regular, then the solution to the above ODE is well-defined and that becomes a 
flow of diffeomorphisms. We subsequently define the precise notion of regularity that is needed.
\begin{definition}[Admissible space \cite{Yo10}]\label{def:Admissible}
A Hilbert space $\Vspace \subset \Smooth^1_0(\domain, \Real^n)$ is admissible if it is (canonically) embedded in 
$\Smooth^1_0(\domain, \Real^n)$ with the $\Vert \cdot \Vert_{1,\infty}$ norm, \ie there exists a constant $C>0$ such that
\[    \Vert \vfield \Vert_{1,\infty}  \leq C \Vert \vfield \Vert_{\Vspace}
     \quad\text{for all $\vfield \in \Vspace$.} \]
In the above, $\Vert \vfield \Vert_{1,\infty} := \Vert \vfield \Vert_{1} + \Vert \Diff\vfield \Vert_{\infty}$ for 
$\vfield \in \Smooth^1_0(\domain, \Real^n)$.
\end{definition}

Then a proper space of velocity fields is well-defined as
\begin{equation}\label{eq:FlowSpace}
  \LpSpace^p([0,1], \Vspace) := 
     \Bigl\{ 
       \text{$\velocityfield  : \velocityfield(t,\cdot) \in \Vspace$ 
       and $\Vert \velocityfield \Vert_{ \LpSpace^p([0,1], \Vspace)}  < \infty$ for $1 \leq p \leq \infty$}
     \Bigr\}
\end{equation} 
with the associated norm
\[ \Vert \velocityfield \Vert_{\LpSpace^p([0,1], \Vspace)} :=  
      \biggl( \int_0^1 \bigl\Vert \velocityfield(t,\cdot) \bigr\Vert^p_{\Vspace}\dint t \biggr)^{1/p}.
\]  
For short, let $\Xspace{p}$ denote $\LpSpace^p([0,1], \Vspace)$. 
Note that $\Xspace{2}$ is a Hilbert space with inner product
\[
\langle \velocityfield, \velocityfieldother\rangle_{\Xspace{2}} = \int_0^1 \bigl\langle\velocityfield(t,\cdot), \velocityfieldother(t,\Cdot) \bigr\rangle_{\Vspace}\dint t \quad\text{for $\velocityfield, \velocityfieldother \in \Xspace{2}$}.
\]

\begin{remark}\label{re:remark1}
A useful case is when $\Vspace$ is a \ac{RKHS} with a symmetric and positive-definite 
reproducing kernel. Then $\Vspace$  is an admissible Hilbert space  \cite{BrHo15}. In the rest of the paper, 
the space of vector fields is selected as an admissible \ac{RKHS} for the advantages of 
sufficient smoothness and fast computability \cite{ChOz18}.
\end{remark}

Furthermore, a flow of diffeomorphisms can be generated via an admissible velocity field, 
which is stated as the following theorem.

\begin{theorem}[\cite{Yo10,BrHo15}]\label{thm:FlowDiffeo}
Let $\Vspace$ be an admissible Hilbert space and $\velocityfield \in \Xspace{2}$ be a velocity field. 
Then the ODE in \cref{eq:FlowEq} admits a unique solution 
$\diffeo^{\velocityfield} \in \Smooth^1_0([0,1] \times \domain, \domain)$, such that 
for $t\in [0,1]$, the mapping $\diffeo_t^{\velocityfield} : \domain \rightarrow \domain$ is 
a $\Smooth^1$-diffeomorphism on $\domain$.
\end{theorem}

Let us define
\begin{equation}\label{eq:SubLieGroupDefine}
 \LieGroup_{\Vspace} := \Bigl\{ \diffeo : \diffeo=\gelement{0,1}{\velocityfield}
      \text{ for some $\velocityfield \in \Xspace{2}$} 
   \Bigr\},
\end{equation}
where 
\begin{equation}\label{eq:FlowRelation}
\gelement{s,t}{\velocityfield} := \diffeo_t^{\velocityfield} \circ (\diffeo_s^{\velocityfield})^{-1} \quad\text{for $0 \leq t,s \leq 1$}
\end{equation} 
and $\diffeo_t^{\velocityfield}$ denotes the solution to the ODE in \cref{eq:FlowEq} with 
given $\velocityfield \in \Xspace{2}$. For $\diffeo_0^{\velocityfield} = \Id$, 
by \cref{eq:FlowRelation} we know 
\begin{equation}\label{eq:diff_trans}
\diffeo_t^{\velocityfield} = \diffeo_{0,t}^{\velocityfield}, \quad (\diffeo_t^{\velocityfield})^{-1} = \diffeo_{t,0}^{\velocityfield}.
\end{equation} 
Next several important properties about $\LieGroup_{\Vspace}$ are stated as follows.

\begin{theorem}[\cite{Yo10,BrHo15}]\label{thm:LieGroupProperty}
Let $\Vspace$ be an admissible Hilbert space, $\LieGroup_{\Vspace}$ be defined in \cref{eq:SubLieGroupDefine},
and $\LieGroupMetric{\Vspace} : \LieGroup_{\Vspace} \times \LieGroup_{\Vspace} \rightarrow \Real_+$ be defined as
\begin{equation}\label{eq:DiffMetric1}
  \LieGroupMetric{\Vspace}(\diffeo,\diffeoother) := 
  \inf_{\substack{\velocityfield \in \Xspace{2} \\ \diffeoother = \diffeo \circ \gelement{0,1}{\velocityfield}}}\, 
     \Vert \velocityfield \Vert_{\Xspace{2}}
  \quad\text{for $\diffeo,\diffeoother \in \LieGroup_{\Vspace}$.}
\end{equation}
Then $\LieGroup_{\Vspace}$ is a group for the composition of functions, and $\LieGroup_{\Vspace}$ is a complete 
metric space under the metric $\LieGroupMetric{\Vspace}$. 
For each $\diffeo, \diffeoother \in \LieGroup_{\Vspace}$, there 
exists $\velocityfield \in \Xspace{2}$ satisfying $\diffeoother = \diffeo \circ \gelement{0,1}{\velocityfield}$, \ie $\LieGroupMetric{\Vspace}(\diffeo,\diffeoother)  = \Vert \velocityfield \Vert_{\Xspace{2}}$.
\end{theorem}

This distance can then be used as a regularization term for image registration via the following LDDMM formulation:
\begin{equation}\label{eq:LDDMM_form_0}
\min_{\diffeo \in \LieGroup_{\Vspace}}  \bigl\Vert \diffeo . I_0 - I_1 \bigr\Vert^2_{\LpSpace^2(\domain)} + \mu\, \LieGroupMetric{\Vspace}^2(\Id,\diffeo),
\end{equation}
where $I_0, I_1 \in \LpSpace^2(\domain)$ are two given images, and $\mu$ is a nonnegative regularization parameter.

The minimum in \cref{eq:DiffMetric1} is reached and then it is shown in \cite[Lemma 11.3]{Yo10} that the previous formulation is equivalent to the following variational model, with a regularization term defined on velocity fields instead of diffeomorphisms:
\begin{equation}\label{eq:LDDMM_form}
\begin{split}
&\min_{\velocityfield \in \Xspace{2}} \bigl\Vert \diffeo_{0, 1}^{\velocityfield} . I_0 - I_1 \bigr\Vert^2_{\LpSpace^2(\domain)} + \mu \int_{0}^{1} \bigl\Vert \velocityfield(t,\cdot) \bigr\Vert^2_{\Vspace}\dint t  \\
&\quad\,\text{s.t.  $\diffeo_{0,1}^{\velocityfield}$ solves \ac{ODE} \cref{eq:FlowEq} at time $t=1$.}
\end{split}
\end{equation}

Hence, the regularization term for image registration by \ac{LDDMM} is formulated as 
\begin{equation}\label{eq:RegTerm_t1}
 \RegFunc(\diffeo) := \LieGroupMetric{\Vspace}^2(\Id,\diffeo) = \int_{0}^{1} \bigl\Vert \velocityfield(t,\cdot) \bigr\Vert^2_{\Vspace}\dint t,
\end{equation}
where the above $\velocityfield$ is an existing minimum 
for $\LieGroupMetric{\Vspace}(\Id,\diffeo)  = \Vert \velocityfield \Vert_{\Xspace{2}}$ such 
that $\velocityfield \in \Xspace{2}$ satisfying $\diffeo = \Id \circ \gelement{0,1}{\velocityfield}$ 
(see \cref{thm:LieGroupProperty}).

\subsection{Spatiotemporal reconstruction with \ac{LDDMM}}\label{sec:LDDMM_based_model}

We assume that the temporal deformation $\diffeo_t^{\velocityfield}$ in \cref{eq:VarReg_2} is 
generated by the flow equation \cref{eq:FlowEq} as in \ac{LDDMM}. 
According to \cref{thm:FlowDiffeo}, the generated flow $\diffeo_t^{\velocityfield}$ is 
diffeomorphism on $\domain$ if the velocity field $\velocityfield \in \Xspace{2}$. 
Consequently, combining \cref{thm:LieGroupProperty} with \cref{eq:RegTerm_t1} implies that 
the shape regularization $\RegFunc_2$ for the temporal deformation $\diffeo_t^{\velocityfield}$ 
in \cref{eq:VarReg_2} can be designed as 
\begin{equation}\label{eq:RegTerm_tt}
\RegFunc_2(\diffeo_t^{\velocityfield}) := \int_{0}^{t} \bigl\Vert \velocityfield(\tau,\cdot) \bigr\Vert^2_{\Vspace}\dint \tau.
\end{equation}
Note that by \cref{thm:LieGroupProperty}, the above $\velocityfield$ is also a minimum 
for $\LieGroupMetric{\Vspace}(\Id,\diffeo_1^{\velocityfield})  = \Vert \velocityfield \Vert_{\Xspace{2}}$.

Using \cref{eq:diff_trans}, we have $\RegFunc_2(\diffeo_t^{\velocityfield}) = \RegFunc_2(\diffeo_{0,t}^{\velocityfield})$. 
Considering the general framework in \cref{sec:VariatRegCT}, as a special form of \cref{eq:VarReg_2}, 
the new variational  model for joint image reconstruction and motion estimation for
spatiotemporal imaging becomes 
\begin{equation}\label{eq:VarReg_LDDMM_2}
\begin{split}
 &\min_{\substack{\template \in \RecSpace \\ \velocityfield \in \Xspace{2}}} \int_{0}^{1} \left[\DataDisc\bigl( \ForwardOp_t(\diffeo_{0,t}^{\velocityfield} . \template), \data(t,\Cdot) \bigr)  + 
   \mu_2\int_{0}^{t} \bigl\Vert \velocityfield(\tau,\cdot) \bigr\Vert^2_{\Vspace}\dint \tau \right] \dint t  + \mu_1 \RegFunc_1(\template)  \\
 & \quad\,\, \text{s.t.  $\diffeo_{0,t}^{\velocityfield}$ solves \ac{ODE} \cref{eq:FlowEq}.}
  \end{split}
\end{equation}   
By simply changing the order of integration for the second term in \cref{eq:VarReg_LDDMM_2}, 
we obtain the equivalent formulation:
\begin{equation}\label{eq:VarReg_LDDMM_2_equiv}
\begin{split}
 &\min_{\substack{\template \in \RecSpace \\ \velocityfield \in \Xspace{2}}} \int_{0}^{1} \left[\DataDisc\bigl( \ForwardOp_t(\diffeo_{0,t}^{\velocityfield} . \template), \data(t,\Cdot) \bigr)  + 
   \mu_2 (1-t) \bigl\Vert \velocityfield(t,\cdot) \bigr\Vert^2_{\Vspace} \right] \dint t  + \mu_1 \RegFunc_1(\template)  \\
 & \quad\,\, \text{s.t.  $\diffeo_{0,t}^{\velocityfield}$ solves \ac{ODE} \cref{eq:FlowEq}.}
  \end{split}
\end{equation} 
The model \cref{eq:VarReg_LDDMM_2} is termed \emph{time-continuous version} of the proposed model. Furthermore, the above model 
can be restated as \ac{PDE}-constrained optimal control formulation, which is given in the following theorem. 
\begin{theorem}\label{thm:EquivalencePDEOptProb_2}
Let $\template \in \RecSpace$ and $\signal \colon [0,1] \times \domain \to \Real$ be defined as 
\begin{equation}\label{eq:Orbit_c1_2}
  \signal(t,\Cdot) := \diffeo_{0,t}^{\velocityfield} . \template \quad\text{for $0 \leq t \leq 1$}, 
\end{equation}
where $\diffeo_{0,t}^{\velocityfield}$ is a diffeomorphism on $\domain$ given by \cref{eq:FlowRelation}.
Assume furthermore that $\signal(t, \Cdot) \in \RecSpace$, and $\RecSpace$ is a sufficiently smooth space. 
Then, \cref{eq:VarReg_LDDMM_2} with the group action given by geometric deformation 
in \cref{eq:GeometricDeformation} is equivalent to 
\begin{equation}\label{eq:LDDMMPDEConstrained_2}
\begin{split}
 &\min_{\substack{\signal(t,\Cdot) \in \RecSpace \\ \velocityfield \in \Xspace{2}}} \int_{0}^{1} \left[\DataDisc\Bigl( \ForwardOp_t\bigr(\signal(t,\Cdot)\bigl), \data(t,\Cdot) \Bigr)  
  +  \mu_2\int_{0}^{t} \bigl\Vert \velocityfield(\tau,\cdot) \bigr\Vert^2_{\Vspace}\dint \tau \right] \dint t + \mu_1 \RegFunc_1\bigl(\signal(0,\Cdot)\bigr) \\
 & \quad\,\, \text{s.t.  $\partial_t \signal(t, \Cdot) + \bigl\langle \grad \signal(t,\cdot), \velocityfield(t,\cdot) \bigr\rangle_{\Real^n} = 0$.}
  \end{split}
\end{equation}  
With the group action given by mass-preserving deformation in \cref{eq:MassPreservedDeform},  
then \cref{eq:VarReg_LDDMM_2}  is equivalent to 
\begin{equation}\label{eq:LDDMMPDEConstrained_mp_2}
\begin{split}
 &\min_{\substack{\signal(t,\Cdot) \in \RecSpace \\ \velocityfield \in \Xspace{2}}} \int_{0}^{1} \left[\DataDisc\Bigl( \ForwardOp_t\bigl(\signal(t,\Cdot)\bigr), \data(t,\Cdot) \Bigr)  
  +  \mu_2\int_{0}^{t} \bigl\Vert \velocityfield(\tau,\cdot) \bigr\Vert^2_{\Vspace}\dint \tau \right] \dint t + \mu_1 \RegFunc_1\bigl(\signal(0,\Cdot)\bigr) \\
 & \quad\,\, \text{s.t. $\partial_t \signal(t, \Cdot) +  \grad\cdot\bigl(\signal(t, \Cdot)\, \velocityfield(t, \Cdot) \bigr)= 0$.}
  \end{split}
\end{equation}  
\end{theorem}
\begin{proof}
First we consider the geometric deformation in \cref{eq:GeometricDeformation}, \ie \cref{eq:Orbit_c1_2} reads as  
\begin{equation}\label{eq:Orbit_c2}
  \signal(t,\Cdot) = \template \circ (\diffeo_{0,t}^{\velocityfield})^{-1}  \quad\text{for $0 \leq t \leq 1$.}
\end{equation}
Obviously, $\signal(0,\Cdot) = \template$ and $\signal(1,\Cdot) = \template \circ (\gelement{0,1}{\velocityfield})^{-1}$. 
Furthermore, applying the variable transformation for \cref{eq:Orbit_c2}, we get 
\begin{equation}\label{eq:Orbit_rewritten}
  \signal\bigl(t, \gelement{0,t}{\velocityfield}\bigr) = \template \quad\text{for $0 \leq t \leq 1$.}
\end{equation}
Differentiating \cref{eq:Orbit_rewritten} \wrt time t leads to 
\[
\partial_t \signal\bigl(t, \gelement{0,t}{\velocityfield}\bigr) 
   + \bigl\langle 
         \nabla \signal\bigl(t,\gelement{0,t}{\velocityfield}\bigr), 
         \velocityfield\bigl(t,\gelement{0,t}{\velocityfield}\bigr) 
      \bigr\rangle_{\Real^n} = 0.
\]
Then the \ac{PDE} constraint in \cref{eq:LDDMMPDEConstrained_2} is obtained. 
Hence a solution to \cref{eq:VarReg_LDDMM_2} generates a solution to \cref{eq:LDDMMPDEConstrained_2}.  

We now consider the reverse implication, \ie demonstrate that a solution to \cref{eq:LDDMMPDEConstrained_2} 
also solves  \cref{eq:VarReg_LDDMM_2}. Suppose that $\signal$ and $\velocityfield$ 
solve \cref{eq:LDDMMPDEConstrained_2}. 
Define the diffeomorphism $\diffeoother_t$ that solves the \ac{ODE} \cref{eq:FlowEq} with the above given $\velocityfield$.  
Since $\signal$ satisfies \ac{PDE} constraint in \cref{eq:LDDMMPDEConstrained_2}, 
considering $t \mapsto \signal\bigl(t, \diffeoother_t\bigr)$, we have 
\[  
\frac{d}{dt} \signal(t, \diffeoother_t) =  \partial_t \signal(t, \diffeoother_t) 
      + \bigl\langle \grad \signal(t,\diffeoother_t), \velocityfield(t,\diffeoother_t) \bigr\rangle_{\Real^n} = 0.
\]
Hence,  $t \mapsto \signal\bigl(t, \diffeoother_t\bigr)$ is constant so in particular we have  
\[  \signal\bigl(t, \diffeoother_t\bigr) \equiv \signal\bigl(0, \diffeoother_0\bigr) = \signal(0, \Cdot). \]
Let $\signal(0, \Cdot)$ be the template $\template$ and $\diffeoother_t$ be $\gelement{0,t}{\velocityfield}$. 
Then $\signal\bigl(t, \Cdot\bigr) = \template \circ (\diffeo_{0,t}^{\velocityfield})^{-1}$.
Hence a solution to \cref{eq:LDDMMPDEConstrained_2} also produces a solution to \cref{eq:VarReg_LDDMM_2}.

Using a mass-preserving deformation \cref{eq:MassPreservedDeform} as group action in \cref{eq:Orbit_c1_2} results in 
\begin{equation}\label{eq:Orbit_cMP}
  \signal(t,\Cdot) = \bigl\vert \Diff\bigl((\diffeo_{0,t}^{\velocityfield})^{-1}\bigr) \bigr\vert \template \circ (\diffeo_{0,t}^{\velocityfield})^{-1} \quad\text{for $0 \leq t \leq 1$.}
\end{equation}
We then get that $\signal(0,\Cdot) = \template$ and 
$\signal(1,\Cdot) = \bigl\vert \Diff\bigl((\gelement{0,1}{\velocityfield})^{-1}\bigr) \bigr\vert \,\template \circ (\gelement{0,1}{\velocityfield})^{-1}$.  
The symmetry of the mass-preserving property furthermore yields 
\begin{equation}\label{eq:PDEConstrained_rewritten_2}
  \bigl\vert \Diff(\gelement{0,t}{\velocityfield}) \bigr\vert \,\signal(t, \Cdot) \circ \gelement{0,t}{\velocityfield}  
    = \template \quad\text{for $0 \leq t \leq 1$.}
\end{equation}
Finally, differentiating \cref{eq:PDEConstrained_rewritten_2} \wrt $t$ leads to 
the constraint in \cref{eq:LDDMMPDEConstrained_mp_2}. 
Hence, a minimizer of \cref{eq:VarReg_LDDMM_2} with the group action given by \cref{eq:MassPreservedDeform} 
is also a minimizer of \cref{eq:LDDMMPDEConstrained_mp_2}.  Similar to the case of geometric deformation, 
it is not difficult to prove the reverse implication. 
\end{proof}

The above equivalent formulation makes it easier for us to compare our proposed approach 
against \ac{PDE} based ones, such as those based on optical flow \cite{BuDiSch18}. 
More details are provided in \cref{sec:LDDMM_based_model_related}.

\subsection{Comparison with existing approaches}\label{sec:LDDMM_based_model_related}

In this section, the mathematical comparison will be made among the proposed model \cref{eq:VarReg_LDDMM_2} 
and several existing approaches (\ie optical flow based model, diffeomorphic motion models).

\subsubsection{Comparison to optical flow based model} \label{sec:optical_flow}

Recently, an optical flow based variational model was proposed 
for joint motion estimation and image reconstruction in spatiotemporal 
imaging, which is called joint \ac{TV}-\ac{TV} optical flow model in \cite{BuDiSch18}. 
The approach is formulated as a \ac{PDE}-constrained optimal control problem, 
so we can compare it to our approach using the reformulation 
in  \cref{eq:LDDMMPDEConstrained_2}.

Since the optical flow based model is set up in terms of the brightness constancy equation, 
this points to using the geometric deformation in \cref{eq:GeometricDeformation} 
as a group action in \cref{eq:Orbit_c1_2}, i.e., we assume \cref{eq:Orbit_c2} holds.  
The optical flow based approach reads as 
\begin{equation}\label{eq:OptFlowPDEConstrained}
\begin{split}
 &\min_{\substack{\signal(t,\Cdot) \in \RecSpace \\ \velocityfield(t,\Cdot) \in BV(\Omega)}} \int_{0}^{1} \left[\DataDisc\Bigl( \ForwardOp_t\bigl(\signal(t,\Cdot)\bigr), \data(t,\Cdot) \Bigr)  
  + \mu_1 \RegFunc_1\bigl(\signal(t,\Cdot)\bigr) +  \mu_2\bigl\Vert \velocityfield(t, \Cdot) \bigr\Vert_{BV}\right]  \dint t \\
 & \quad\ \ \ \text{s.t. $\partial_t \signal(t, \Cdot) + \bigl\langle \grad \signal(t,\Cdot), \velocityfield(t,\Cdot) \bigr\rangle_{\Real^n} = 0$,}
  \end{split}
\end{equation}  
where $\|\cdot\|_{BV}$ is the \ac{TV} semi-norm in the space of functions with bounded variation (BV)
\[
BV(\Omega) := \left\{u\in L^1(\Omega): \|u\|_{BV} = \sup_{\varphi \in \Smooth^1_0(\domain, \Real^n), \|\varphi\|_{\infty} \leq 1} \int_{\Omega} u\Div \varphi \dint x < \infty\right\}.
\]
Note that $\bigl\Vert \velocityfield(t, \Cdot) \bigr\Vert_{BV}$ denotes the sum of the \ac{TV} semi-norm of all the elements in $\velocityfield(t, \Cdot)$ \cite{AuKo02}.  

It is easy to see that the constraints in \cref{eq:LDDMMPDEConstrained_2} and \cref{eq:OptFlowPDEConstrained} 
are equivalent. Hence, the geometric deformation is equivalent to using an the optical flow constraint 
or the brightness constancy equation \cite{BuDiSch18}.

By comparison, the primary distinction between \cref{eq:LDDMMPDEConstrained_2} 
and \cref{eq:OptFlowPDEConstrained} relates to the selection of the regularization 
term \wrt vector field $\velocityfield(t, \cdot)$. In model \cref{eq:OptFlowPDEConstrained}, 
its selection is \ac{TV} semi-norm. Hence, the space of vector fields is included 
in $BV(\Omega)$, which allows for a vector field that is a piecewise-constant vector-valued function 
distributed on $\Omega$. By contrast in model \cref{eq:LDDMMPDEConstrained_2}, 
the space of vector fields is included in an admissible Hilbert space. Hence, 
the vector field is a sufficiently smooth vector-valued function distributed on $\Omega$. 
This guarantees an elastic diffeomorphic deformation, which is close to the physical 
mechanism to some extent \cite{GiRuBu12,BuMoRu13}.

In addition to the above, both approaches also differ in the selection of regularization 
term $\RegFunc_1$. In \cref{eq:LDDMMPDEConstrained_2} one only poses restriction on 
the initial image $\signal(0, \Cdot)$, whereas in \cref{eq:OptFlowPDEConstrained} the whole 
time trajectory $t \mapsto \signal(t, \Cdot)$ is regularized. Hence, \cref{eq:LDDMMPDEConstrained_2} 
has a simpler structure which is also beneficial in implementation. 

\subsubsection{Compared with diffeomorphic motion models} \label{sec:diff_motion_model}

To characterize the optimality conditions for \cref{eq:VarReg_LDDMM_2} we introduce the 
notation
\begin{equation}\label{eq:LDDMM_match_short}
\DataDisc_{\data_t}\bigl(\signal\bigr)  :=  \DataDisc\Bigl( \ForwardOp_t\bigl(\signal\bigr), \data(t,\Cdot) \Bigr)
\end{equation}
for $\signal \in \RecSpace$ with given $\data(t,\Cdot) \in \DataSpace$. 
By \cref{thm:energy_functional_derivative_2}  in \cref{sec:Optimality_conditions}, 
the optimal velocity field in \cref{eq:VarReg_LDDMM_2} satisfies 
\begin{equation}\label{eq:optimal_velocity_field}
\velocityfield(t,\Cdot) =  
   \frac{1}{2\mu_2 (1-t)}  \int_{t}^{1}\Koperator\Bigl( \nabla (\diffeo_{0,t}^{\velocityfield} . \template) \bigl\vert \Diff(\gelement{t,\tau}{\velocityfield}) \bigr\vert \grad\DataDisc_{g_{\tau}} (\diffeo_{0,\tau}^{\velocityfield} . \template) ( \gelement{t,\tau}{\velocityfield})  \Bigr) \dint \tau 
  \end{equation}
for $0 \leq t < 1$. The above optimal velocity field can be seen as the average of 
the integrand \wrt the time integral. The $\velocityfield(1, \Cdot)$ is well-defined at $t=1$ as
\begin{equation}\label{eq:opt_vel_field_t1}
\velocityfield(1,\cdot) = \frac{1}{2\mu_2}\Koperator\bigl( \nabla (\diffeo_{0,1}^{\velocityfield} . \template)  \grad\DataDisc_{g_1} (\diffeo_{0,1}^{\velocityfield} . \template ) \bigr). 
\end{equation}
By \cref{eq:Energy_functional_gradient_template_2} and \cref{eq:OptCond_2} we get at $t=0$ that 
\begin{equation}\label{eq:opt_vel_field_t0}
\velocityfield(0,\cdot) =  \frac{1}{2\mu_2} \Koperator\biggl( \grad \template \int_{0}^{1} \bigl\vert \Diff(\gelement{0,t}{\velocityfield}) \bigr\vert \grad\DataDisc_{g_t} (\diffeo_{0,t}^{\velocityfield} . \template ) ( \gelement{0, t}{\velocityfield} )  \dint t \biggr).
\end{equation}
Note that the optimal velocity field in \cref{eq:VarReg_LDDMM_2} is \emph{averagely distributed \wrt time $t$}, 
and also \emph{non-vanishing neither on the initial nor on the end time point}. 

A diffeomorphic motion model was proposed for 4D \ac{CT} image reconstruction 
in \cite{HiSzWaSaJo12} that is based on the \ac{LDDMM} growth model \cite{GrMi07}. 
The related time-continuous model reads as 
\begin{equation}\label{eq:DiffeoPDEConstrained}
\begin{split}
 &\min_{\substack{\template \in \RecSpace \\ \velocityfield \in \Xspace{2}}} \int_{0}^{1} \Bigl[ \DataDisc\bigl( \ForwardOp_t(\gelement{0,t}{\velocityfield} . \template ), \data(t,\Cdot) \bigr)  
  +   \mu_2\bigl\Vert \velocityfield(t,\cdot) \bigr\Vert^2_{\Vspace} \Bigr]  \dint t \\
 & \quad\,\, \text{s.t.  $\diffeo_{0,t}^{\velocityfield}$ solves ODE \cref{eq:FlowEq}.}
  \end{split}
\end{equation}   
Compared to \cref{eq:VarReg_LDDMM_2}, the above approach 
neglects the regularization term $\RegFunc_1$ absolutely. 
Another difference relates to the selection on the shape regularization $\RegFunc_2$. 
In \cref{eq:DiffeoPDEConstrained}, the shape regularization is a uniformly weighted term 
on $\bigl\Vert \velocityfield(t,\cdot) \bigr\Vert^2_{\Vspace}$. 
In contrast, that is a non-uniformly weighted term in \cref{eq:VarReg_LDDMM_2} 
(see \cref{eq:VarReg_LDDMM_2_equiv} for more clear), which fulfils more weights on the previous time.

\begin{remark}\label{rem:explain_more_reg}
Note that in \cref{eq:VarReg_LDDMM_2}, we regularize the 
velocity field more at the beginning, and which is relevant because the template is selected at the 
initial time, and the more beginning of the velocity field, the more influence on the whole geodesic trajectory.   
\end{remark}

For further comparison, by \cref{thm:energy_functional_derivative_2}, it is easy to see that 
the $\Xspace{2}$-norm minimizer of  \cref{eq:DiffeoPDEConstrained} \wrt variations of the velocity field satisfies
\begin{equation}\label{eq:DiffeoPDEConstrainedMin1}
\velocityfield(t,\cdot) = \frac{1}{2\mu_2} \int_{t}^{1}  \Koperator\Bigl(\grad (\gelement{0,t}{\velocityfield} . \template )  \bigl\vert \Diff(\gelement{t,\tau}{\velocityfield}) \bigr\vert  \grad \DataDisc_{g_{\tau}}( \gelement{0,\tau}{\velocityfield} .  \template )(\gelement{t,\tau}{\velocityfield}) \Bigr) \dint \tau 
\end{equation}
for $0\leq t \leq 1$ and $\mu_2 > 0$. In addition, the minimizer \wrt variations of the template satisfies
\begin{equation}\label{eq:DiffeoPDEConstrainedMin2}
\int_0^1 \bigl\vert \Diff(\gelement{0,t}{\velocityfield}) \bigr\vert \grad \DataDisc_{g_t}( \gelement{0,t}{\velocityfield} . \template )(\gelement{0,t}{\velocityfield}) \dint t = 0.
\end{equation}
Combining \cref{eq:DiffeoPDEConstrainedMin1} and \cref{eq:DiffeoPDEConstrainedMin2}, we immediately have 
\[
\velocityfield(0,\cdot) = \velocityfield(1,\cdot) = 0.
\]
Clearly, the optimal velocity field that minimizes \cref{eq:DiffeoPDEConstrained} 
is \emph{vanishing both on the initial and end time points}. It is not difficult to see 
from \cref{eq:DiffeoPDEConstrainedMin1} that the optimal velocity field is {\it  not averagely distributed \wrt time $t$}.

On the other hand, the following model can be seen as another diffeomorphic motion model. 
\begin{equation}\label{eq:VarReg_LDDMM}
\begin{split}
 &\min_{\substack{\template \in \RecSpace \\ \velocityfield \in \Xspace{2}}} \int_{0}^{1} \left[\DataDisc\bigl( \ForwardOp_t(\diffeo_{0,t}^{\velocityfield} . \template), \data(t,\Cdot) \bigr)  
  + \mu_1 \RegFunc_1(\diffeo_{0,t}^{\velocityfield} . \template) +  \mu_2\int_{0}^{t} \bigl\Vert \velocityfield(\tau,\cdot) \bigr\Vert^2_{\Vspace}\dint \tau \right] \dint t \\
 & \quad\,\, \text{s.t.  $\diffeo_{0,t}^{\velocityfield}$ solves \ac{ODE} \cref{eq:FlowEq}.}
  \end{split}
\end{equation}  
We consider the regularization term on the whole $\diffeo_t . \template$ instead of only on $\template$ in \cref{eq:VarReg_2}.  
By \cref{thm:energy_functional_derivative_2},  its optimal velocity field for $0 \leq t < 1$ satisfies
\begin{equation}\label{eq:OptCond_diffeoModel1}
\velocityfield(t, \Cdot) = \frac{1}{2\mu_2 (1-t)}\int_{t}^{1} \Koperator\Bigl( \nabla (\diffeo_{0,t}^{\velocityfield} . \template)  \bigl\vert \Diff(\gelement{t,\tau}{\velocityfield}) \bigr\vert \grad\MatchingFunctionalX_{g_{\tau}} (\diffeo_{0,\tau}^{\velocityfield} . \template)( \gelement{t,\tau}{\velocityfield}) \Bigr) \dint \tau 
\end{equation}
where with fixed $\data(t,\Cdot) \in \DataSpace$, 
\begin{equation*}
\MatchingFunctionalX_{\data_t} \bigl(\signal\bigr) := \DataDisc\Bigl(\ForwardOp_t\bigl(\signal\bigr), \data(t,\Cdot) \Bigr) + \mu_1 \RegFunc_1\bigl(\signal\bigr)
\end{equation*}
for $\signal \in \RecSpace$. 
And its optimal template satisfies
\begin{equation}\label{eq:OptCond_diffeoModel2}
\int_0^1 \bigl\vert \Diff(\gelement{0,t}{\velocityfield}) \bigr\vert \grad\MatchingFunctionalX_{g_t} \bigl(\diffeo_{0,t}^{\velocityfield} . \template\bigr)\bigl( \gelement{0,t}{\velocityfield}\bigr) \dint t = 0. 
\end{equation}
Evidently, the above optimal velocity field is also a time average of the integrand.
Even though $\velocityfield(1, \Cdot)$ is well-defined at $t=1$ as
\begin{equation*}
\velocityfield(1,\cdot) 
= \frac{1}{2\mu_2}\Koperator\bigl( \grad (\diffeo_{0,1}^{\velocityfield} . \template)  \grad\MatchingFunctionalX_{g_1} (\diffeo_{0,1}^{\velocityfield} . \template ) \bigr), 
\end{equation*}
by \cref{eq:OptCond_diffeoModel1} and \cref{eq:OptCond_diffeoModel2} at $t=0$ we have 
\[
\velocityfield(0,\cdot) = 0.
\] 
Hence, the optimal velocity field in \cref{eq:VarReg_LDDMM} is \emph{averagely distributed \wrt time $t$}, 
but \emph{vanishing on the initial time point}.

The above analysis points to several advantages that comes with using \cref{eq:VarReg_LDDMM_2} 
over alternative approaches.

\section{Numerical implementation}\label{sec:computed_method}

Let us first present the time-discretized version of the proposed model. 

\subsection{Time-discretized version}\label{sec:time_discretized_version}

Basically, the data set is gained by gating method in the manner of uniformly discretized time, \ie based on a uniform 
partition of $[0,1]$ as $\{t_i\}_{i=0}^{N}$, and $t_i = i/N$ for $0 \leq i \leq 1$. We refer to this as the gating grid and  
the time-discretized version of the general spatiotemporal inverse problem in \cref{eq:InvProb} reads as 
\begin{equation}\label{eq:InvProb_discrete}  
     \data(t_i,\Cdot) = \ForwardOp_{t_i}\bigl(\signal(t_i,\Cdot)\bigr) + \noisedata(t_i,\Cdot). 
\end{equation}
Thus the time-discretized version of \cref{eq:VarReg_LDDMM_2} becomes 
\begin{equation}\label{eq:time-discretized_VarReg_LDDMM_2}
\begin{split}
 &\min_{\substack{\template \in \RecSpace \\ \velocityfield \in \Xspace{2}}} \frac{1}{N}\sum_{i=1}^{N} \left[\DataDisc\bigl( \ForwardOp_{t_i}(\diffeo_{0,t_i}^{\velocityfield} . \template), \data(t_i,\Cdot) \bigr)  + 
   \mu_2\int_{0}^{t_i} \bigl\Vert \velocityfield(\tau,\cdot) \bigr\Vert^2_{\Vspace}\dint \tau \right]  + \mu_1 \RegFunc_1(\template)  \\
 & \quad\,\, \text{s.t.  $\diffeo_{0,t}^{\velocityfield}$ solves \ac{ODE} \cref{eq:FlowEq}.}
  \end{split}
\end{equation}  
\begin{remark}\label{re:otherchoose}
The time-discretized version \cref{eq:time-discretized_VarReg_LDDMM_2} can be also written such 
that the image in the first gate is the template: 
\begin{equation*}
\begin{split}
 &\min_{\substack{\template \in \RecSpace \\ \velocityfield \in \Xspace{2}}} \frac{1}{N+1}\sum_{i=0}^{N} \left[\DataDisc\bigl( \ForwardOp_{t_i}(\diffeo_{0,t_i}^{\velocityfield} . \template), \data(t_i,\Cdot) \bigr)  + 
   \mu_2\int_{0}^{t_i} \bigl\Vert \velocityfield(\tau,\cdot) \bigr\Vert^2_{\Vspace}\dint \tau \right]  + \mu_1 \RegFunc_1(\template)  \\
 & \quad\,\, \text{s.t.  $\diffeo_{0,t}^{\velocityfield}$ solves \ac{ODE} \cref{eq:FlowEq}.}
  \end{split}
\end{equation*}  
\end{remark}

Since \cref{eq:time-discretized_VarReg_LDDMM_2} contains highly coupled arguments, 
it is difficult to jointly solve for the template $\template$ and the velocity field $\velocityfield$. 
A relaxed method is to compute $\template$ and $\velocityfield$ in an intertwined manner.
More precisely, a fixed velocity field $\velocityfield$ yields the flow of 
diffeomorphisms $\diffeo_t$ through \ac{ODE} \cref{eq:FlowEq}. 
Hence, the spatiotemporal reconstruction problem \cref{eq:time-discretized_VarReg_LDDMM_2} reduces 
to the following modified static image reconstruction problem:
\begin{equation}\label{eq:VarReg_LDDMM_template_time_discrete2}
 \min_{\template \in \RecSpace}\frac{1}{N}\sum_{i=1}^{N}\DataDisc\bigl( \ForwardOp_{t_i}(\diffeo_{0,t_i}^{\velocityfield} . \template), \data(t_i,\Cdot) \bigr) + \mu_1 \RegFunc_1(\template).
 \end{equation} 
Conversely, if the template $\template$ is fixed then \cref{eq:time-discretized_VarReg_LDDMM_2} 
boils down to a sequentially indirect image registration problem where we seek the velocity 
field $\velocityfield$ from time-series data that are indirect observations of the target:
\begin{equation}\label{eq:VarReg_LDDMM_deformation_time_discrete2}
\begin{split}
 & \min_{\velocityfield \in \Xspace{2}} \frac{1}{N}\sum_{i=1}^{N}\left[ \DataDisc\bigl( \ForwardOp_{t_i}(\diffeo_{0,t_i}^{\velocityfield} . \template), \data(t_i,\Cdot) \bigr)  + 
   \mu_2\int_{0}^{t_i} \bigl\Vert \velocityfield(\tau,\cdot) \bigr\Vert^2_{\Vspace}\dint \tau \right]  \\
 &\quad\,\,  \text{s.t.  $\diffeo_{0,t}^{\velocityfield}$ solves ODE \cref{eq:FlowEq}.}
  \end{split}
  \end{equation}
We solve \cref{eq:time-discretized_VarReg_LDDMM_2} by alternately solving 
for \cref{eq:VarReg_LDDMM_template_time_discrete2} and \cref{eq:VarReg_LDDMM_deformation_time_discrete2}. 

If the data fidelity term is designed 
as the squared $\LpSpace^2$-norm in \cref{eq:DataMatching_L2} 
and the spatial regularization is selected as the \ac{TV} functional in \cref{eq:TVReg}, then  
\cref{eq:time-discretized_VarReg_LDDMM_2} is written as  
\begin{equation}\label{eq:VarReg_LDDMM_time-discretized}
\begin{split}
 &\min_{\substack{\template \in \RecSpace \\ \velocityfield \in \Xspace{2}}} \frac{1}{N}\sum_{i=1}^{N}\left[ \|\ForwardOp_{t_i}\bigl(\diffeo_{0,t_i}^{\velocityfield} . \template\bigr) - \data(t_i,\Cdot)\|_2^2  + 
   \mu_2\int_{0}^{t_i} \bigl\Vert \velocityfield(\tau,\cdot) \bigr\Vert^2_{\Vspace}\dint \tau \right]  + \mu_1 \|\nabla \template\| \\
 & \quad\,\, \text{s.t.  $\diffeo_{0,t}^{\velocityfield}$ solves ODE \cref{eq:FlowEq}.}
  \end{split}
\end{equation}  
Note that one may of course choose other data fidelity and spatial regularization terms as indicated in \cref{sec:VariatRegCT}. 
Correspondingly, \cref{eq:VarReg_LDDMM_template_time_discrete2} becomes
\begin{equation}\label{eq:VarReg_LDDMM_template_time_discrete}
 \min_{\template \in \RecSpace}\GoalFunctionalV_{\velocityfield}(\template)  := \frac{1}{N}\sum_{i=1}^{N}\|\ForwardOp_{t_i}\bigl(\diffeo_{0,t_i}^{\velocityfield} . \template\bigr) - \data(t_i,\Cdot)\|_2^2 + \mu_1 \|\nabla \template\|,
 \end{equation} 
and \cref{eq:VarReg_LDDMM_deformation_time_discrete2} reads as
\begin{equation}\label{eq:VarReg_LDDMM_deformation_time_discrete}
\begin{split}
 & \min_{\velocityfield \in \Xspace{2}} \GoalFunctionalV_ {\template}(\velocityfield)  := \frac{1}{N} \sum_{i=1}^{N}\left[ \|\ForwardOp_{t_i}\bigl(\diffeo_{0,t_i}^{\velocityfield} . \template\bigr) - \data(t_i,\Cdot)\|_2^2  + 
   \mu_2\int_{0}^{t_i} \bigl\Vert \velocityfield(\tau,\cdot) \bigr\Vert^2_{\Vspace}\dint \tau \right]  \\
 &\quad\,\,  \text{s.t.  $\diffeo_{0,t}^{\velocityfield}$ solves the ODE \cref{eq:FlowEq},}
  \end{split}
  \end{equation}
where $\GoalFunctionalV_{\velocityfield} \colon \RecSpace \to \Real$ and 
$\GoalFunctionalV_ {\template} \colon \Xspace{2} \to \Real$. 
We then solve \cref{eq:VarReg_LDDMM_time-discretized} by alternately solving 
for \cref{eq:VarReg_LDDMM_template_time_discrete} and \cref{eq:VarReg_LDDMM_deformation_time_discrete}, \ie
\begin{equation}\label{eq:Alternating}
\begin{cases}
  \template^{k+1} := \text{ solution to \cref{eq:VarReg_LDDMM_template_time_discrete} with $\velocityfield = \velocityfield^k$,} \\[0.5em]
  \velocityfield^{k+1} := \text{ solution to \cref{eq:VarReg_LDDMM_deformation_time_discrete} with $\template = \template^{k+1}$.} 
\end{cases}
\end{equation}

\subsection{Template reconstruction}\label{subsec:LinearizedSplitBregman}

Just as in \cite{BlNaRa12,BrPaOeKuKa12,LiZhZhGa15,BuDiSch18}, the temporal evolution operator in 
spatiotemporal imaging can be given by the geometric deformation as   
$\diffeo_{0,t} . \template := \template \circ \diffeo_{t,0}^{-1} = \template \circ \diffeo_{0,t}$. 
We will henceforth consider this setting. 
We here describe the steps underlying the implementation for solving the static image 
reconstruction problem in \cref{eq:VarReg_LDDMM_template_time_discrete}. The resulting 
gradient descent scheme is summarized in \cref{algo:GDSB_4DCT}.
The resulting template, which is obtained assuming a given velocity field, 
gives the images in all gates by deforming it under the flow of diffeomorphisms. 

The optimization problem in \cref{eq:VarReg_LDDMM_time-discretized} is a 
non-smooth \ac{TV}-$\ell_2$ minimization. We modify this non-smooth problem 
into a smooth one as
\begin{equation}\label{eq:template_recon_CT}
\min\limits_{\template\in \RecSpace}\frac{1}{N}\sum_{i=1}^{N}\bigl\Vert \ForwardOp_{t_i}\bigl(\template \circ \diffeo_{t_i,0}^{\velocityfield}\bigr) - \data(t_i,\Cdot) \bigr\Vert^2_{2} + \mu_1 \int_{\domain}\vert \nabla \template(x) \vert_{1, \epsilon}\dint x,
\end{equation}
where $\vert \nabla \template(x) \vert_{1, \epsilon} = \sqrt{\sum_i\bigl(\partial_i\template(x)\bigr)^2 + \epsilon}$ with 
$\epsilon > 0$ small, \eg $\epsilon = 10^{-12}$. This is a frequently used modification 
for \ac{TV} regularization in image reconstruction \cite{SiKaPa06,ChXu12,ChPan17}.
Then \cref{eq:template_recon_CT} can be solved by the following gradient descent scheme: 
\begin{multline}\label{eq:soc4}
\template^{k+1} = \template^k - \alpha^k\biggl(\frac{2}{N}\sum_{i=1}^{N}\bigl\vert\Diff\bigl(\diffeo_{0,t_i}^{\velocityfield}\bigr)\bigr\vert\ForwardOpAdjoint_{t_i}\Bigl(\ForwardOp_{t_i}(\template^k \circ \diffeo_{t_i,0}^{\velocityfield}) - \data(t_i,\Cdot) \Bigr)(\diffeo_{0,t_i}^{\velocityfield}) \\
+ \mu_1\grad^{\,\ast}\Bigl(\frac{\grad \template^k}{\vert \nabla \template^k \vert_{1, \epsilon}}\Bigr)\biggr)
\end{multline}
with $\alpha^k$ the stepsize for the $k$-th iteration, where $\mathcal{T}$ is assumed to be linear, 
and $\mathcal{T}^{\ast}$ denotes its adjoint operator.

Note that several convex optimization techniques for solving non-smooth problems in 
static image reconstruction, such as the algorithms 
in \cite{sidky2012convex,ChXu13,NiFe15,ChXu16,BaCo17} and references therein, can be 
used to solve \cref{eq:VarReg_LDDMM_template_time_discrete} without modification. 
Basically such kind of methods need to introduce 
more auxiliary variables or parameters than the above algorithm. To optimize the whole 
problem \cref{eq:VarReg_LDDMM_time-discretized} efficiently, we employ the 
iterative scheme \cref{eq:soc4} to solve this subproblem.

\subsubsection{Computing diffeomorphic deformations}
\label{subsubsection:CompDiffDeform}

Updating $\template^{k+1}$ requires computing diffeomorphic deformations $\diffeo_{t_i,0}^{\velocityfield}$ 
and $\diffeo_{0,t_i}^{\velocityfield}$ for $1 \leq i \leq N$. 
 
By definition, $\gelement{s,t}{\velocityfield}$ solves the flow equation 
\begin{equation}\label{eq:BasicODEInv_0_2}
\begin{cases}
  \partial_t \varphi(t,x) = \velocityfield\bigl(t,\varphi(t,x)\bigr) & \\[0.5em]
  \varphi(s,x) = x & 
 \end{cases}
 \quad\text{for $x\in \domain$ and $0\leq s, t \leq 1$,} 
\end{equation}
where $s$ is a fixed time point. Integrating \wrt time $t$ in \cref{eq:BasicODEInv_0_2} yields 
\begin{equation}\label{eq:forwarddeformation}
  \gelement{s,t}{\velocityfield} = \Id + \int_s^t \velocityfield\bigl(\tau,\gelement{s,\tau}{\velocityfield}\bigr)\dint \tau \quad\text{for $0\leq  t \leq 1$.} 
  \end{equation}
  
The time interval $[0, 1]$ is subdivided uniformly into $MN$ parts thereby 
forming a discretized time grid that is given as $\tau_j = j/(MN)$ for $j = 0, 1, \ldots, MN$. 
Evidently, $\tau_{iM} = t_i$ for $i = 0, 1, \ldots, N$, so each subinterval $[t_i, t_{i+1}]$ is 
segmented into $M$ even parts. The $M$ is named as a factor of discretized time degree.  If $M=1$, then $\tau_i = t_i$, 
namely, the discretized time grid is consistent with the gating grid. Needs to be pointed out is 
that the different subintervals of gating grid can be discretized adaptively according to the degree of motions. 

Within a short-time interval one can approximate the diffeomorphic deformation 
with linearized deformations \cite{OkChDoRaBa16}. More precisely, 
let $s = \tau_j$, $t = \tau_{j-1}$ and $\tau_{j+1}$ in \cref{eq:forwarddeformation}, then the expressions for 
small deformations $\gelement{\tau_i, \tau_{i-1}}{\velocityfield}$ and $\gelement{\tau_i, \tau_{i+1}}{\velocityfield}$ 
can be approximated by
\begin{equation}\label{eq:deformation_approx2}
\gelement{\tau_j, \tau_{j-1}}{\velocityfield} \approx  \Id - \frac{1}{MN}\velocityfield(\tau_j,\Cdot),
\end{equation}
and 
\begin{equation}\label{eq:deformation_approx1}
\gelement{\tau_j, \tau_{j+1}}{\velocityfield} \approx  \Id + \frac{1}{MN}\velocityfield(\tau_j,\Cdot).
\end{equation}
Moreover, \cref{eq:FlowRelation} implies 
that $\gelement{\tau_j, 0}{\velocityfield} = \gelement{\tau_{j-1}, 0}{\velocityfield} \circ \gelement{\tau_j, \tau_{j-1}}{\velocityfield}$, 
which combined with \cref{eq:deformation_approx2} yields 
\begin{equation}\label{eq:deformation_approx3}
\gelement{\tau_j,0}{\velocityfield} \approx  \gelement{\tau_{j-1},0}{\velocityfield} \circ \Bigl(\Id - \frac{1}{MN}\velocityfield(\tau_j,\Cdot)\Bigr)  \quad\text{for $j = 1, 2, \ldots, MN$.}
\end{equation}
Next, \cref{eq:deformation_approx3} yields the following estimate for $\template \circ \gelement{\tau_j,0}{\velocityfield}$:
\begin{equation}\label{eq:deformation_approx7}
\template \circ \gelement{\tau_j,0}{\velocityfield} \approx  \bigl(\template \circ \gelement{\tau_{j-1},0}{\velocityfield}\bigr) \circ \Bigl(\Id - \frac{1}{MN}\velocityfield(\tau_j,\Cdot)\Bigr)
\end{equation}
for $j = 1, 2, \ldots, MN$ and with $\template \circ \gelement{\tau_0,0}{\velocityfield} = \template$. 
Similarly, \cref{eq:FlowRelation} also implies   
 $\gelement{\tau_j, t_i}{\velocityfield} = \gelement{\tau_{j+1}, t_i}{\velocityfield} \circ \gelement{\tau_j, \tau_{j+1}}{\velocityfield}$
for $i \geq 1$, which combined with \cref{eq:deformation_approx1} gives the following approximation:
\begin{equation}\label{eq:deformation_approx4}
\gelement{\tau_j, t_i}{\velocityfield} \approx  \gelement{\tau_{j+1}, t_i}{\velocityfield} \circ \Bigl(\Id + \frac{1}{MN}\velocityfield(\tau_j,\Cdot)\Bigr) 
\end{equation}
for $j = iM-1, iM-2, \ldots, 0$ and with $\gelement{t_i,t_i}{\velocityfield} = \Id$.

To summarize, the deformation between two images of adjacent points of discretized time grid is approximately
represented as a linearized deformation. 

\subsubsection{Computing mass-preserving deformations}
The gradient of the data fidelity term involves the type of mass-preserving deformation in \cref{eq:MassPreservedDeform} as 
\[ \bigl\vert\Diff\bigl(\diffeo_{0,t_i}^{\velocityfield}\bigr)\bigr\vert \ForwardOpAdjoint_{t_i}\Bigl(\ForwardOp_{t_i}\bigl(\template \circ \diffeo_{t_i,0}^{\velocityfield}\bigr) - \data(t_i,\Cdot) \Bigr) \circ\gelement{0,t_i}{\velocityfield} \quad \text{for $i \geq 1$.}
\]
Starting with the Jacobian determinant, by \cref{eq:deformation_approx4} we get 
\begin{equation}\label{eq:deformation_approx5}
\bigl\vert \Diff(\gelement{\tau_j, t_i}{\velocityfield})\bigr\vert \approx \Bigl(1 + \frac{1}{MN} \Div\velocityfield(\tau_j,\Cdot)\Bigr) 
 \bigl\vert  \Diff(\gelement{\tau_{j+1},t_i}{\velocityfield})\bigr\vert  \circ \Bigl(\Id + \frac{1}{MN}\velocityfield(\tau_j,\Cdot) \Bigr) 
 \end{equation}
for $j = iM-1, iM-2, \ldots, 0$ and where $\bigl\vert \Diff(\gelement{t_i,t_i}{\velocityfield})\bigr\vert = 1$. 
Next, \cref{eq:deformation_approx4} also yields the following approximation: 
\begin{multline}\label{eq:deformation_approx6}
\ForwardOpAdjoint_{t_i}\Bigl(\ForwardOp_{t_i}\bigl(\template \circ \diffeo_{t_i,0}^{\velocityfield}\bigr) - \data(t_i,\Cdot) \Bigr) \circ\gelement{\tau_j,t_i}{\velocityfield} \\
\approx  \ForwardOpAdjoint_{t_i}\Bigl(\ForwardOp_{t_i}\bigl(\template \circ \diffeo_{t_i,0}^{\velocityfield}\bigr) - \data(t_i,\Cdot) \Bigr) \circ \gelement{\tau_{j+1},t_i}{\velocityfield}  \circ \Bigl(\Id + \frac{1}{MN}\velocityfield(\tau_j,\Cdot) \Bigr) 
\end{multline}
for $j = iM-1, iM-2, \ldots, 0$. For simplicity, let
\begin{equation}\label{eq:h_t_ti}
\eta_{\tau, t}^{\template, \velocityfield} = \bigl\vert \Diff\bigl(\gelement{\tau,t}{\velocityfield}\bigr) \bigr\vert \ForwardOpAdjoint_t\Bigl(\ForwardOp_t\bigl(\template \circ \gelement{t,0}{\velocityfield}\bigr) - \data(t,\Cdot)\Bigr) \circ \gelement{\tau,t}{\velocityfield}.
\end{equation}
Then multiplying \cref{eq:deformation_approx5} by \cref{eq:deformation_approx6}, and using \cref{eq:h_t_ti}, $\eta_{0, t_i}^{\template, \velocityfield}$ for $i \geq 1$ is computed by
\begin{equation}\label{eq:eta_t_ti_update}
 \eta_{\tau_j, t_i}^{\template,\velocityfield} \approx \Bigl(1 + \frac{1}{MN} \Div\velocityfield(\tau_j,\Cdot)\Bigr) 
 \eta_{\tau_{j+1}, t_i}^{\template,\velocityfield}  \circ \Bigl(\Id + \frac{1}{MN}\velocityfield(\tau_j,\Cdot) \Bigr)
 \end{equation}
for $j = iM-1, iM-2, \ldots, 0$ with $\eta_{t_i, t_i}^{\template, \velocityfield} = \ForwardOpAdjoint_{t_i}\Bigl(\ForwardOp_{t_i}\bigl(\template \circ \gelement{t_i,0}{\velocityfield}\bigr) - \data(t_i,\Cdot)\Bigr)$.

Based on the above derivations, the concrete implementation is given as the 
gradient descent scheme in \cref{algo:GDSB_4DCT}.

\begin{algorithm}
\caption{Gradient descent scheme for minimizing $\GoalFunctionalV_ {\velocityfield}(\template)$ in \cref{eq:VarReg_LDDMM_template_time_discrete}}
\label{algo:GDSB_4DCT}
\begin{algorithmic}[1]
\State \emph{Initialize}:
\State $k \gets 0$.
\State $t_i \gets \frac{i}{N}$ for $i = 0, 1, \ldots, N$.
\State $\tau_j \gets \frac{j}{MN}$ for $j = 0, 1, \ldots, MN$. 
\State Given $\velocityfield$.
\State $\template^k \gets \template^0$. Here $\template^0$ is a given initial template.
\State Spatial regularization parameter $\mu_1 > 0$.
\State Error tolerance $\epsilon_{\template} > 0$, stepsize $\alpha^k = \alpha > 0$, and iteration number $K_{\template} > 0$.
\State \emph{Loop}:
\State \quad Compute $\template^k \circ \gelement{\tau_j,0}{\velocityfield}$ for $1 \leq j \leq MN$ by 
\[
\template^k \circ \gelement{\tau_j,0}{\velocityfield} \gets  \bigl(\template^k \circ \gelement{\tau_{j-1},0}{\velocityfield}\bigr) \circ \Bigl(\Id - \frac{1}{MN}\velocityfield(\tau_j,\Cdot)\Bigr)
\]
\quad with $\template^k \circ \gelement{0,0}{\velocityfield} = \template^k$. 
\State \quad Update $\eta_{t_i, t_i}^{\template^k, \velocityfield}$ for $1 \leq i \leq N$ by
\[
\eta_{t_i, t_i}^{\template^k, \velocityfield} \gets \ForwardOpAdjoint_{t_i}\bigl(\ForwardOp_{t_i}(\template^k \circ \gelement{t_i,0}{\velocityfield}) - \data(t_i,\Cdot)\bigr)
\].
\State \quad Compute $\eta_{0, t_i}^{\template^k, \velocityfield}$ for $1 \leq i \leq N$ by
\[
\eta_{\tau_j, t_i}^{\template^k, \velocityfield}    
  \gets  \Bigl(1 + \frac{1}{MN} \Div\velocityfield(\tau_j,\Cdot)\Bigr)  \eta_{\tau_{j+1}, t_i}^{\template^k, \velocityfield} \circ \Bigl(\Id + \frac{1}{MN}\velocityfield(\tau_j,\Cdot) \Bigr) 
 \]
\quad for $j = iM-1, iM-2, \ldots, 0$.
\State \quad Evaluate $\template^{k+1}$ by
\[
\template^{k+1} \gets \template^k - \alpha\biggl(\frac{2}{N}\sum_{i=1}^{N}\eta_{0, t_i}^{\template^k,\velocityfield} 
+ \mu_1\grad^{\,\ast}\Bigl(\frac{\grad \template^k}{\vert \nabla \template^k \vert_{1, \epsilon}}\Bigr)\biggr).
\]
\State \quad \textbf{If} $\bigr\vert \template^{k+1} - \template^k \bigr\vert > \epsilon_{\template}$ and $k<K_{\template}$, then $k \gets k+1$, \textbf{goto} \emph{Loop}.
\State \textbf{Output} $\template^{k+1}$.
\end{algorithmic}
\end{algorithm}

\subsection{Velocity field estimation}
\label{subsec:GradientDescent_velocithfield}

The aim here is to provide an algorithm for solving \cref{eq:VarReg_LDDMM_deformation_time_discrete}, 
which is sequentially indirect image registration. 
We will use a gradient descent scheme of the form 
\begin{equation}\label{eq:gradientflow}
  \velocityfield^{k+1} =  \velocityfield^k -  \beta^k \grad \GoalFunctionalV_{\template}(\velocityfield^k).
\end{equation}
Here $\GoalFunctionalV_{\template} \colon \Xspace{2} \to \Real$ is the objective functional in \cref{eq:VarReg_LDDMM_deformation_time_discrete}, 
$\beta^k$ is the step-size in the $k$-th iteration, and $\grad \GoalFunctionalV_{\template}(\velocityfield)$ is 
calculated by \cref{eq:Energy_functional_time_discretized_gradient}.

The central issue is the computation of $\grad \GoalFunctionalV_{\template}$ and 
the final algorithm for the gradient descent scheme \cref{eq:gradientflow} is 
given in \cref{algo:GradientDescentAlgorithmForFiniteFunctional_0}.

\subsubsection{Computing $\grad \GoalFunctionalV_ {\template}$}\label{subsec:GradientDescent}

Let us first introduce notations:
\begin{equation}\label{eq:Middle_func_deriv_1}
h_{\tau, t}^{\template,\velocityfield} := \begin{cases} 
    \eta_{\tau, t}^{\template, \velocityfield}, \quad 0 \leq \tau \leq t \leq 1,    & \\[0.5em]
    0, \quad t < \tau, &  
   \end{cases} 
\end{equation}
and
\begin{equation}\label{eq:Middle_func_deriv_2}
\velocityfield_{\tau, t} := \begin{cases} 
   \velocityfield(\tau,\Cdot), \quad 0 \leq \tau \leq t \leq 1,    & \\[0.5em]
   0, \quad t < \tau. &  
   \end{cases}
\end{equation}
\Cref{thm:energy_functional_time_discretized_derivative} gives an expression 
for $\grad \GoalFunctionalV_{\template}$ where the kernel 
function $\kernel \colon \domain \times \domain \to \Matrix_{+}^{n \times n}$ is evaluated on 
points that do not move as iteration proceeds. By choosing a translation invariant kernel and points on a regular 
grid in $\domain$, we can use \acs{FFT}-based convolution scheme 
to efficiently evaluate the velocity field at each iteration. 
This is computationally more feasible than letting the kernel depend on points that move in time as in the 
shooting method \cite{MiTrYo06,ViRiRuCo12}.

In what follows, we write out the explicit derivations for 
computing $\grad\GoalFunctionalV_ {\template}(\velocityfield)$. 
As derived in \cref{subsubsection:CompDiffDeform}, $\template \circ \gelement{\tau_j,0}{\velocityfield}$ 
can be approximated by \cref{eq:deformation_approx7}. The key step is now to 
update $h_{\tau_j, t_i}^{\template,\velocityfield}$ for $\{i : t_i \geq \tau_j\}$ 
in \cref{eq:Energy_functional_time_discretized_gradient}. We know, by \cref{eq:Middle_func_deriv_1}, 
\begin{equation}\label{eq:h_eqv_eta}
h_{\tau_j, t_i}^{\template,\velocityfield} = \eta_{\tau_j, t_i}^{\template,\velocityfield} \quad \text{for $t_i \geq \tau_j$}.
\end{equation}
Using \cref{eq:eta_t_ti_update} for $1 \leq i \leq N$ allows us to compute $h_{\tau_j, t_i}^{\template,\velocityfield}$ by 
\begin{equation}\label{eq:h_t_ti_update}
 h_{\tau_j, t_i}^{\template,\velocityfield} \approx \Bigl(1 + \frac{1}{MN} \Div\velocityfield(\tau_j,\Cdot)\Bigr) 
 h_{\tau_{j+1}, t_i}^{\template,\velocityfield}  \circ \Bigl(\Id + \frac{1}{MN}\velocityfield(\tau_j,\Cdot) \Bigr)
 \end{equation}
 for $j = iM-1, iM-2, \ldots, 0$ and with 
 $h_{t_i, t_i}^{\template, \velocityfield} = \ForwardOpAdjoint_{t_i}\Bigl(\ForwardOp_{t_i}(\template \circ \gelement{t_i,0}{\velocityfield}) - \data(t_i,\Cdot)\Bigr)$.
Hence, at $t=\tau_j$, by \cref{eq:Energy_functional_time_discretized_gradient} we get 
 \begin{multline}\label{eq:algo_step3} 
 \grad\GoalFunctionalV_ {\template}(\velocityfield)(\tau_j,x)   \\
     = - \frac{2}{N} \sum_{\{i\geq 1 : t_i \geq \tau_j\}} \biggl[\int_{\domain}\kernel(x,y)\nabla \bigl(\template \circ \gelement{\tau_j,0}{\velocityfield}\bigr)(y) h_{\tau_j, t_i}^{\template,\velocityfield}(y) \dint y 
 - \mu_2\velocityfield(\tau_j, x) \biggr]
  \end{multline}
for $0\leq j \leq MN$ and $x\in \domain$. In particular, for $j = MN$ (\ie $\tau_j = 1$) we have 
 \begin{equation*}
 \grad\GoalFunctionalV_ {\template}(\velocityfield)(1, x)   
     = - \frac{2}{N} \biggl[\int_{\domain}\kernel(x,y)\nabla \bigl(\template \circ \gelement{1, 0}{\velocityfield}\bigr)(y) h_{1, 1}^{\template,\velocityfield}(y) \dint y  - \mu_2\velocityfield(1, x) \biggr].
  \end{equation*}

\begin{remark}\label{rem:consistent}
It is easy to verify that the optimal solution of the time-discretized version of the proposed model 
is consistent with that of the time-continuous one. This is however not the case for the diffeomorphic 
motion model in \cite{HiSzWaSaJo12}. As an example, at $\tau_j =1$, the optimal velocity field of the    
time-discretized problem in \cite{HiSzWaSaJo12} satisfies
\[
\velocityfield(1, x) = \frac{1}{\mu_2}\int_{\domain}\kernel(x,y)\nabla \bigl(\template \circ \gelement{1, 0}{\velocityfield}\bigr)(y) h_{1, 1}^{\template,\velocityfield}(y) \dint y.
\]
However, as derived in \cref{sec:diff_motion_model},  the optimal velocity field at $t =1$ of its  
time-continuous problem satisfies
$\velocityfield(1, x)  = 0$.
This obviously causes inconsistencies and our consistent approach is an advantage compared to the approach in \cite{HiSzWaSaJo12}.  
\end{remark}  

Finally, \cref{algo:GradientDescentAlgorithmForFiniteFunctional_0} outlines the procedure for computing the 
gradient descent scheme \cref{eq:gradientflow} that makes use of the above derivations.
\begin{algorithm}
\caption{Gradient descent scheme for minimizing $\GoalFunctionalV_ {\template}(\velocityfield)$ in \cref{eq:VarReg_LDDMM_deformation_time_discrete}}
\label{algo:GradientDescentAlgorithmForFiniteFunctional_0}
\begin{algorithmic}[1]
\State \emph{Initialize}:
\State $k \gets 0$.
\State $t_i \gets \frac{i}{N}$ for $i = 0, 1, \ldots, N$. 
\State $\tau_j \gets \frac{j}{MN}$ for $j = 0, 1, \ldots, MN$. 
\State Fixed $\template$.
\State $\velocityfield^k(\tau_i) \gets \velocityfield^0(\tau_i)$, where $\velocityfield^0$ is a given initial velocity field.
\State Fixed kernel function $\kernel(\Cdot,\Cdot)$.
\State Shape regularization parameter $\mu_2 > 0$.
\State Error tolerance $\epsilon_{\velocityfield} > 0$, stepsize $\beta^k = \beta > 0$, and maximum iterations $K_{\velocityfield} > 0$.
\State \emph{Loop}:
\State  \quad Update $\template \circ \gelement{\tau_j,0}{\velocityfield^k}$ for $1 \leq j \leq MN$ by 
\[
\template \circ \gelement{\tau_j,0}{\velocityfield^k} \gets  \bigl(\template \circ \gelement{\tau_{j-1},0}{\velocityfield^k}\bigr) \circ \Bigl(\Id - \frac{1}{N}\velocityfield^k(\tau_j,\Cdot)\Bigr) 
\]
\quad with $\template \circ \gelement{0,0}{\velocityfield^k} = \template$.
\State \quad Update $h_{t_i, t_i}^{\template, \velocityfield^k}$ for $1 \leq i \leq N$ by
\[
h_{t_i, t_i}^{\template, \velocityfield^k} \gets \ForwardOpAdjoint_{t_i}\bigl(\ForwardOp_{t_i}(\template \circ \gelement{t_i,0}{\velocityfield^k}) - \data(t_i,\Cdot)\bigr).
\]
\State \quad Compute $h_{\tau_j, t_i}^{\template, \velocityfield^k}$ for $1 \leq i \leq N$ by
\[
h_{\tau_j, t_i}^{\template, \velocityfield^k}    
  \gets  \Bigl(1 + \frac{1}{MN} \Div\velocityfield(\tau_j,\Cdot)\Bigr)  h_{\tau_{j+1}, t_i}^{\template, \velocityfield^k} \circ \Bigl(\Id + \frac{1}{MN}\velocityfield^k(\tau_j,\Cdot) \Bigr) 
 \]
\quad for $j = iM-1, iM-2, \ldots, 0$.
\State \quad Compute $\grad\GoalFunctionalV_{\template}(\velocityfield^k)(\tau_j,\Cdot) $ (using \ac{FFT} to compute the convolution) by 
 \begin{multline*}
 \grad\GoalFunctionalV_ {\template}(\velocityfield^k)(\tau_j,x)   \\
     \gets - \frac{2}{N} \sum_{\{i\geq 1 : t_i \geq \tau_j\}} \biggl[\int_{\domain}\kernel(x,y)\grad \bigl(\template \circ \gelement{\tau_j,0}{\velocityfield^k}\bigr)(y) h_{\tau_j, t_i}^{\template,\velocityfield^k}(y) \dint y 
 - \mu_2\velocityfield^k(\tau_j, x) \biggr]
  \end{multline*}
\quad  for $0 \leq j \leq MN$. 
\State \quad Update $\velocityfield^k(\tau_j, \Cdot)$ for $0 \leq j \leq MN$ by:
\[
 \velocityfield^{k+1}(\tau_j,\Cdot)  \gets \velocityfield^k(\tau_j,\Cdot) -  \beta \grad\GoalFunctionalV_ {\template}(\velocityfield^k)(\tau_j,\Cdot).
\]
\State \quad \textbf{If} $\bigr\vert \velocityfield^{k+1} - \velocityfield^k \bigr\vert > \epsilon_{\velocityfield}$ and $k<K_{\velocityfield}$, then $k \gets k+1$, \textbf{goto} \emph{Loop}.
\State \textbf{Output} $\velocityfield^{k+1}$.
\end{algorithmic}
\end{algorithm}

\subsection{Alternating template reconstruction and velocity field estimation}

As described in the beginning of \cref{sec:computed_method}, we aim to
solve \cref{eq:VarReg_LDDMM_time-discretized} by an iterative scheme where iterates 
for $\template$ and $\velocityfield$ are updated in an alternating manner as in \cref{eq:Alternating}.
Hence, at each step we solve two sub-problems, one for updating $\template$ 
given $\velocityfield$ (\cref{algo:GDSB_4DCT} in \cref{subsec:LinearizedSplitBregman}) 
and the other for updating $\velocityfield$ given 
$\template$ (\cref{algo:GradientDescentAlgorithmForFiniteFunctional_0} 
in \cref{subsec:GradientDescent_velocithfield}). 

The algorithms for solving the two sub-problems are iterative, so there are inner 
iterations for each outer iterative step that update the template and velocity field. 
Our ultimate aim however is to obtain the minimum of the whole 
model \cref{eq:VarReg_LDDMM_time-discretized}, even if solve each subproblem thoroughly, 
we may have no any benefit to arrive at the desirable solution rapidly. 
Hence this motives us to limit the inner iteration number to be one for solving each subproblem. 
The final algorithm for recovering the template and velocity field is presented in the 
following \cref{algo:Alternating_reconstruction}. We further analyze the 
computational complexity of \cref{algo:Alternating_reconstruction}.

\begin{algorithm}
\caption{Alternately minimizing model \cref{eq:VarReg_LDDMM_deformation_time_discrete}}
\label{algo:Alternating_reconstruction}
\begin{algorithmic}[1]
\State \emph{Initialize}:
\State Given $M, N$.
\State $k \gets 0$.
\State $t_i \gets \frac{i}{N}$ for $i = 0, 1, \ldots, N$. This subdivides the time interval $[0,1]$ uniformly into $N$ parts. 
\State $\tau_j \gets \frac{j}{MN}$ for $j = 0, 1, \ldots, MN$. This subdivides the time interval $[0,1]$ uniformly into $MN$ parts. 
\State Fixed kernel function $\kernel(\Cdot,\Cdot)$.
\State Given regularization parameters $\mu_1, \mu_2 > 0$.
\State $\template^k \gets \template^0$, where the template is initialized.
\State $\velocityfield^k(\tau_i) \gets 0$, where the velocity field is initialized to a zero velocity field.
\State Error tolerances $\epsilon_{\template}, \epsilon_{\velocityfield} > 0$, stepsizes $\alpha^k = \alpha > 0, \beta^k = \beta > 0$, and maximum iteration number $K > 0$.
\State \emph{Loop}:
\State  \quad Let $\velocityfield = \velocityfield^k$. Perform Lines 10-13 in \cref{algo:GDSB_4DCT}. Output $\template^{k+1}$.
\State  \quad Let $\template = \template^{k+1}$. Perform Lines 11-15 in \cref{algo:GradientDescentAlgorithmForFiniteFunctional_0}. Output $\velocityfield^{k+1}$.
\State \quad \textbf{If} $\bigr\vert \velocityfield^{k+1} - \velocityfield^k \bigr\vert > \epsilon_{\velocityfield}$ or $\bigr\vert \template^{k+1} - \template^k \bigr\vert > \epsilon_{\template}$, and $k<K$,  \\
\quad then $k \gets k+1$, \textbf{goto} \emph{Loop}.
\State \textbf{Output} $\template^{k+1}$, $\velocityfield^{k+1}$.
\end{algorithmic}
\end{algorithm}

\paragraph{Complexity analysis} The complexity analysis, including computational cost and space complexity, 
is presented for \cref{algo:Alternating_reconstruction}. Since the main part of each iteration of \cref{algo:Alternating_reconstruction} 
is located on lines 12-13 (actually lines 10-13 in \cref{algo:GDSB_4DCT} 
and lines 11-15 in \cref{algo:GradientDescentAlgorithmForFiniteFunctional_0}), we restrict our complexity  analysis 
to these parts. For ease of description, suppose that $\domain \subset \Real^2$ and the size of the image to 
be reconstructed is $n\times n$ pixels. 

On line 10 of \cref{algo:GDSB_4DCT} and line 11 of \cref{algo:GradientDescentAlgorithmForFiniteFunctional_0}, we need to 
update $\template \circ \gelement{\tau_j,0}{\velocityfield}$ for $j = 1, \ldots, MN$. Moreover, each of them should be 
used to compute the gradient of the objective functional on line 14 
of \cref{algo:GradientDescentAlgorithmForFiniteFunctional_0}, so they need to be stored at hand. 
Hence, in these two steps, the computational cost is $O(n^2MN)$ and the space complexity is $O(n^2MN)$. 

For line 11 of \cref{algo:GDSB_4DCT} and line 12 of \cref{algo:GradientDescentAlgorithmForFiniteFunctional_0}, 
the $\eta_{t_i, t_i}^{\template, \velocityfield}$ (\ie $h_{t_i, t_i}^{\template, \velocityfield}$) need to be 
updated and then stored for $i = 1, \ldots, N$. The computational cost is $O(n^2 N_d N)$, where $N_d$ is the 
number of data points. Actually, the $N_d$ is at least proportional to the size of $n$, 
which is often $\sqrt{2} n N_v$ with $N_v$ denoting the number of views.
Hence, the computational cost is as much as $O(n^3NN_v)$. 
Since the calculation for the forward and backward 
projections is on the fly, the required space is not too demanding.

Furthermore, on line 12 of \cref{algo:GDSB_4DCT} and 
line 13 of \cref{algo:GradientDescentAlgorithmForFiniteFunctional_0}, for $i = 1, \ldots, N$, 
the $\eta_{\tau_j, t_i}^{\template, \velocityfield}$ (\ie $h_{\tau_j, t_i}^{\template, \velocityfield}$) need to be 
updated and stored for $j$ from $iM-1$ to $0$, then are used to compute the gradient of the 
objective functional for each time point on line 14 of \cref{algo:GradientDescentAlgorithmForFiniteFunctional_0}. 
Therefore, the computational cost is  $O(n^2MN^2)$. For lines 11-12 of \cref{algo:GDSB_4DCT} and 
lines 12-13 of \cref{algo:GradientDescentAlgorithmForFiniteFunctional_0}, the space complexity is $O(n^2MN^2)$. 

For line 13 of \cref{algo:GDSB_4DCT}, we need to update $\template$ once, the computational cost 
is $O(n^2N)$ and the space complexity is $O(n^2N)$.  At each time point, the \ac{FFT} is used to compute the gradient of the 
objective functional on line 14 of \cref{algo:GradientDescentAlgorithmForFiniteFunctional_0}. 
Hence the computational cost for this line is $O(MN^2n^2\log n)$. For line 15 
of \cref{algo:GradientDescentAlgorithmForFiniteFunctional_0}, we need to update a vector field 
at each time point. Since a vector field would take twice more memory than a scalar field on 2D domain, 
we spend twice more computational cost to update that. Even so, the computational cost 
is $O(n^2MN)$ and the space complexity is $O(n^2MN)$.

In summary, for \cref{algo:Alternating_reconstruction}, 
the computational cost is at least $O(n^3N)$ and the space complexity is $O(n^2MN^2)$. 

\section{Numerical experiments}
\label{sec:numerical _experiments}

In this section, the proposed method for joint image reconstruction and motion estimation is applied to parallel beam tomography 
with very sparse or highly noisy data in spatiotemporal (\eg 2D + time) imaging. We use the intensity-preserving group action to 
consider the involved deformations. Although this is not a full evaluation, it nevertheless illustrates the 
performance of the proposed method. The numerical implementation is partially based on Operator 
Discretization Library (\href{http://github.com/odlgroup/odl}{http://github.com/odlgroup/odl}).

The forward operator $\ForwardOp_t \colon \RecSpace \to \DataSpace$ is realized by 2D Radon transforms, namely,
\begin{equation*}
R(f)(\omega, x) = \int_{\Real} f(x + s\omega)\dint s  \quad \text{for $\omega \in S^1$ and $x \in \omega^{\bot}$},
\end{equation*}
where $R$ denotes Radon transform, $S^1$ is the unit circle and $(\omega, x)$ determines a line on $\Real^2$ with direction $\omega$ through $x$.

Moreover, consider $\Vspace$ as the space of vector fields that is a \ac{RKHS} with a reproducing kernel 
represented by symmetric and positive definite Gaussian 
function $\kernel \colon \domain \times \domain \to \Matrix_{+}^{2 \times 2}$ given as 
\begin{equation}\label{eq:KernelEq}
  \kernel(x,y) := 
      \exp\Bigl(-\dfrac{1}{2 \sigma^2} \Vert x-y \Vert_2^2 \Bigr)
      \begin{pmatrix} 
          1  & 0 \\
          0  & 1
      \end{pmatrix} 
\quad\text{for $x,y \in \Real^2$.}
\end{equation}
The $\sigma >0$ acts as a kernel width. 

The images of all gates are supported on $\domain$. For image in each gate, 
the noise-free data per view is measured by evaluating the 2D parallel beam 
scanning geometry. Then the additive Gaussian white noise 
at varying levels is added onto the noise-free data, which leads to the noise data.  
As in \cite{ChOz18}, the noise level in data is quantified in terms of \ac{SNR} defined in logarithmic decibel (dB).

\subsection{Test suites and results}\label{sec:results}

The test suites seek to assess the performance against different noise levels, and the 
sensitivity against various selections of regularization parameters $\mu_1$, $\mu_2$, and kernel width $\sigma$. 
We also compare the proposed method to \ac{TV}-based static reconstruction method.

\subsubsection{Test suite 1: Overview performance}

Here we consider a test for evaluating the overview performance. This test uses a multi-object phantom with five gates (\ie $N = 5$). 
The used phantom is shown in the last row of \cref{Test_suite_1:multi_object_phantom}, which is taken from \cite{ChOz18}. 

The image in each gate is consisting of six separately star-like objects with grey-values over $[0, 1]$, which is digitized 
using $438 \times 438$ pixels. The images of all gates are supported on a fixed rectangular 
domain $\Omega = [-16, 16] \times [-16, 16]$. For image in each gate, the noise-free data per view 
is measured by the 2D parallel beam scanning geometry with even 620 bins, which is supported on the range of $[-24, 24]$.  
For gate $i\,(1 \le i \le N)$, the scanning views are distributed 
on $[(i-1)\pi /36, \pi + (i-1) \pi/36]$ uniformly, and the view number is 12. 
Then three different levels of additive Gaussian white noise are added onto the noise-free data. 
The resulting \ac{SNR} are about $4.71$dB, $7.7$dB, and $14.67$dB, respectively. 
To make clear, we show the noise-free and noise projection data of the first view for each. 
\begin{figure}[htbp]
\centering
\begin{minipage}[t]{0.33\textwidth}%
     \centering
     \includegraphics[trim=30 15 30 40, clip, width=\textwidth]{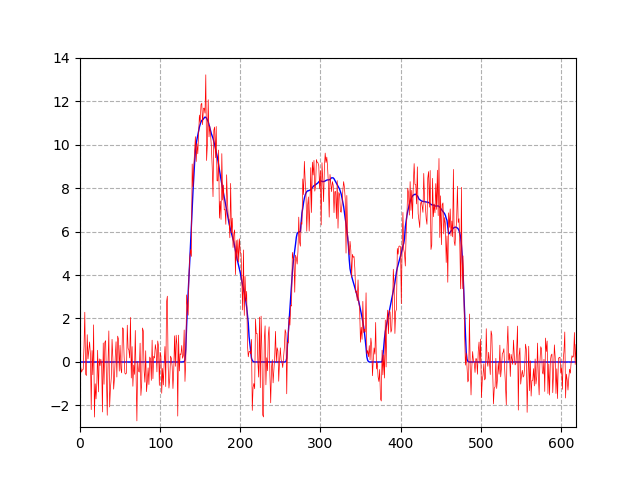}
     \vskip-0.25\baselineskip
   \end{minipage}%
   \hfill
   \begin{minipage}[t]{0.33\textwidth}%
     \centering
    \includegraphics[trim=30 15 30 40, clip, width=\textwidth]{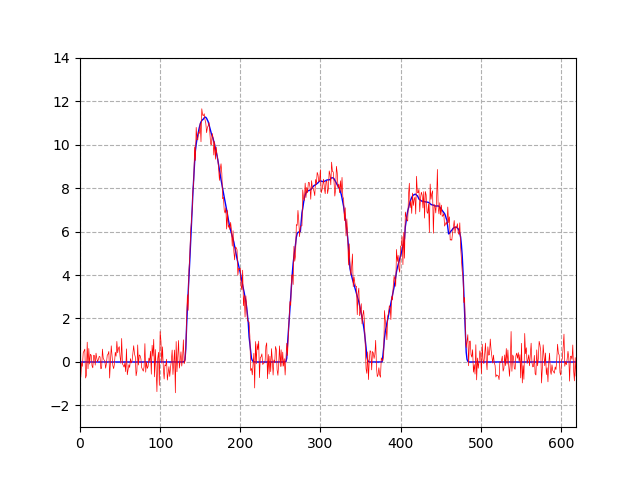}
     \vskip-0.25\baselineskip
   \end{minipage}%
   \hfill
   \begin{minipage}[t]{0.33\textwidth}%
     \centering
    \includegraphics[trim=30 15 30 40, clip, width=\textwidth]{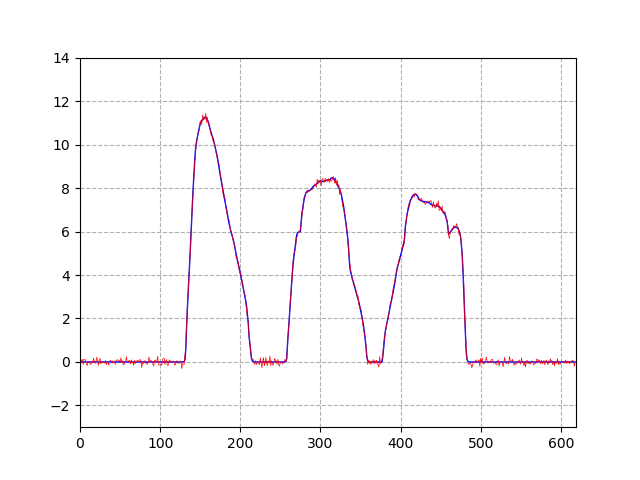}
     \vskip-0.25\baselineskip
   \end{minipage}%
   \caption{Data at the first view for Gate 1. The left, middle, and right figures show data at the first view for different noise levels $4.71$dB, $7.7$dB, and $14.67$dB, respectively. The blue smooth curve is noise-free data, and the red jagged curve is noisy data.}
\label{Test_suite_1:different_noise_levels}
\end{figure}

The factor of discretized time degree is $M = 2$, which is defined in \cref{subsubsection:CompDiffDeform}. 
The kernel width is selected to $\sigma = 2$. 
The gradient stepsizes are set as $\alpha = 0.01$ and $\beta = 0.05$, respectively. 
First we apply \cref{algo:GDSB_4DCT} to obtain an initial template image after 50 
iterations, then use \cref{algo:Alternating_reconstruction} to solve the proposed model. 
Note that the above iteration number is not unchangeable, just needs enough to gain an appropriately 
initial template for \cref{algo:Alternating_reconstruction}.

The regularization parameters ($\mu_1, \mu_2$) are selected 
as $(0.05, 10^{-7})$ for data noise level $4.71$dB, $(0.025, 10^{-7})$ for data noise level $7.7$dB,
and $(0.01, 10^{-7})$ for data noise level $14.67$dB, respectively. 
The lower \ac{SNR}, the lager value of $\mu_1$. The maximum iteration number is 
set to be $200$. The reconstructed results are shown in \cref{Test_suite_1:multi_object_phantom}. 
It is clear that the reconstructed images (rows~1--3) are close to the corresponding ground truth, 
even though the data \ac{SNR} is very low.  
\begin{figure}[htbp]
\centering
\begin{minipage}[t]{0.2\textwidth}%
     \centering
     \includegraphics[trim=75 25 60 40, clip, width=\textwidth]{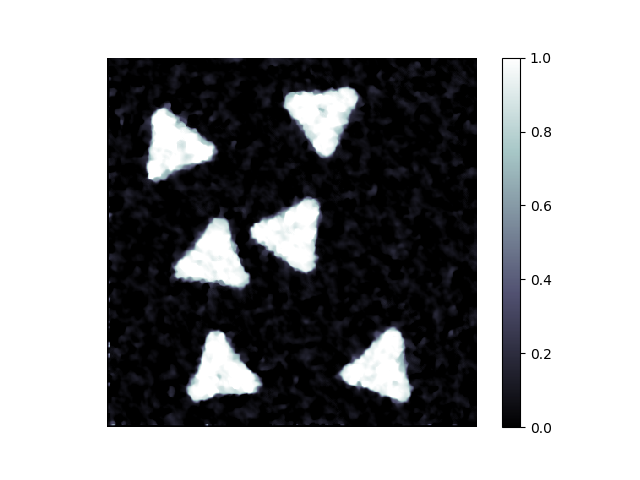}
     \vskip-0.25\baselineskip
   \end{minipage}%
   \hfill
   \begin{minipage}[t]{0.2\textwidth}%
     \centering
    \includegraphics[trim=75 25 60 40, clip, width=\textwidth]{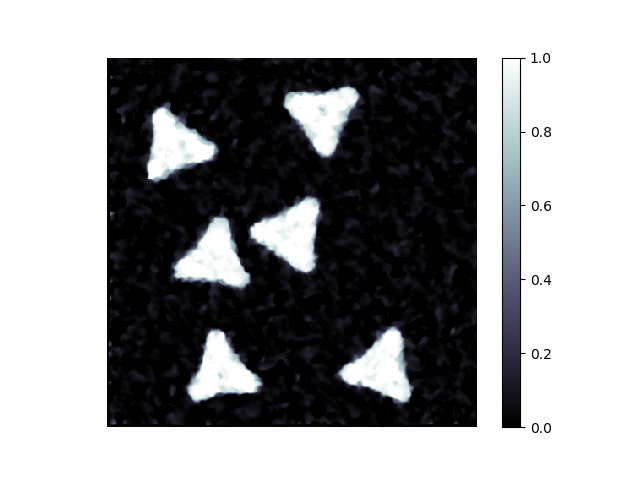}
     \vskip-0.25\baselineskip
   \end{minipage}%
   \hfill
   \begin{minipage}[t]{0.2\textwidth}%
     \centering
     \includegraphics[trim=75 25 60 40, clip, width=\textwidth]{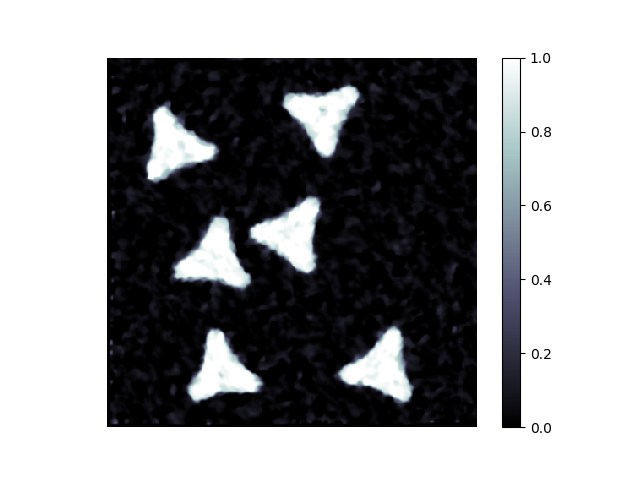}
     \vskip-0.25\baselineskip
   \end{minipage}%
      \hfill
   \begin{minipage}[t]{0.2\textwidth}%
     \centering
     \includegraphics[trim=75 25 60 40, clip, width=\textwidth]{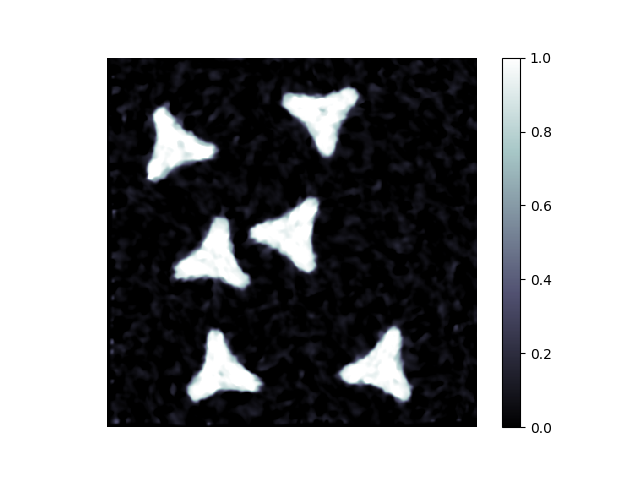}     
     \vskip-0.25\baselineskip
   \end{minipage}%
   \hfill
   \begin{minipage}[t]{0.2\textwidth}%
     \centering
    \includegraphics[trim=75 25 60 40, clip, width=\textwidth]{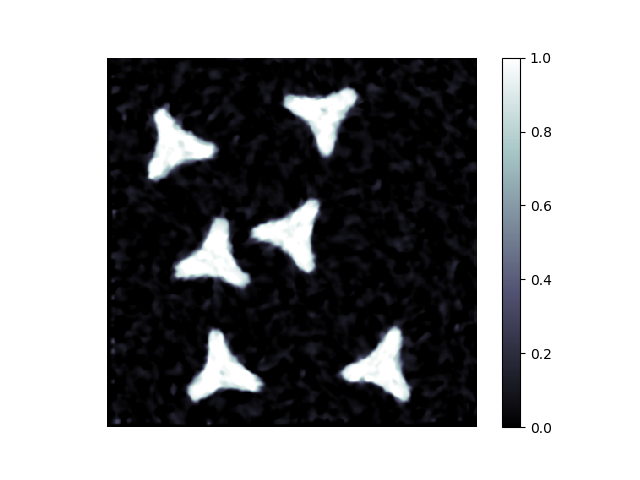}
     \vskip-0.25\baselineskip
   \end{minipage}%
\par\medskip      
\begin{minipage}[t]{0.2\textwidth}%
     \centering
     \includegraphics[trim=75 25 60 40, clip, width=\textwidth]{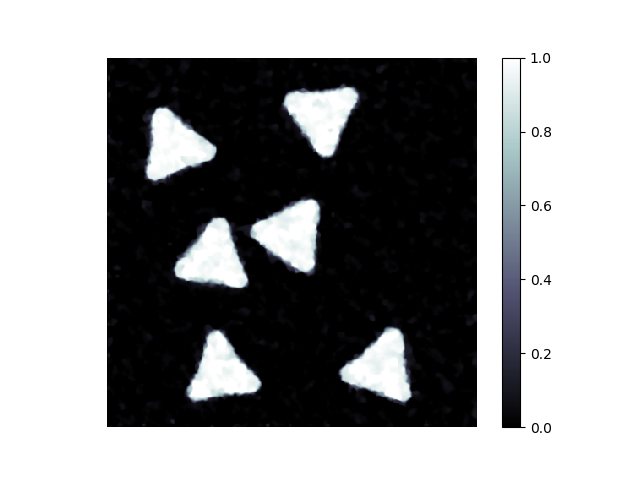}
     \vskip-0.25\baselineskip
   \end{minipage}%
   \hfill
   \begin{minipage}[t]{0.2\textwidth}%
     \centering
    \includegraphics[trim=75 25 60 40, clip, width=\textwidth]{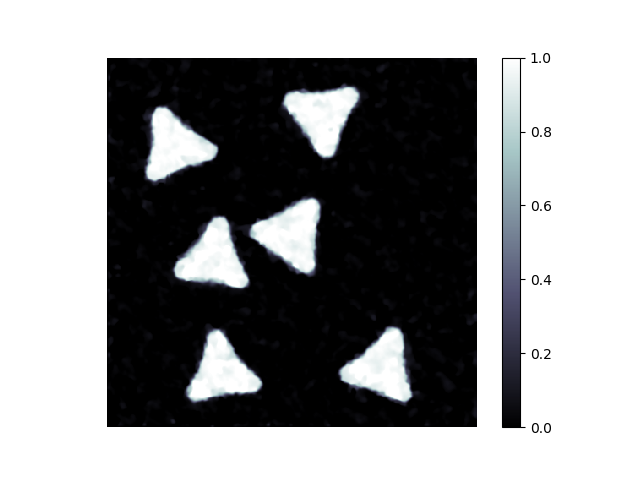}
     \vskip-0.25\baselineskip
   \end{minipage}%
   \hfill
   \begin{minipage}[t]{0.2\textwidth}%
     \centering
     \includegraphics[trim=75 25 60 40, clip, width=\textwidth]{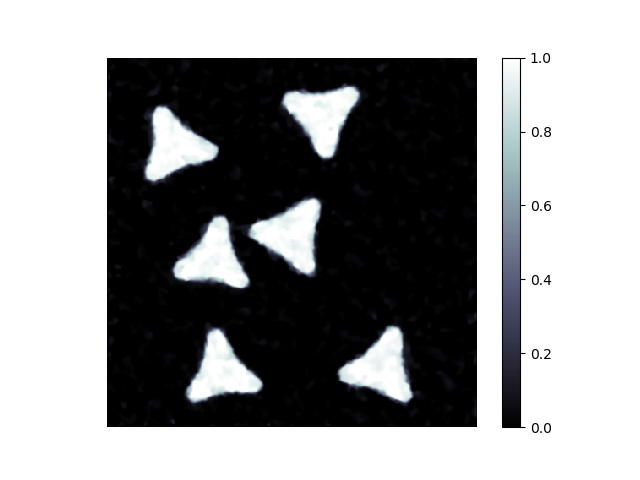}
     \vskip-0.25\baselineskip
   \end{minipage}%
      \hfill
   \begin{minipage}[t]{0.2\textwidth}%
     \centering
     \includegraphics[trim=75 25 60 40, clip, width=\textwidth]{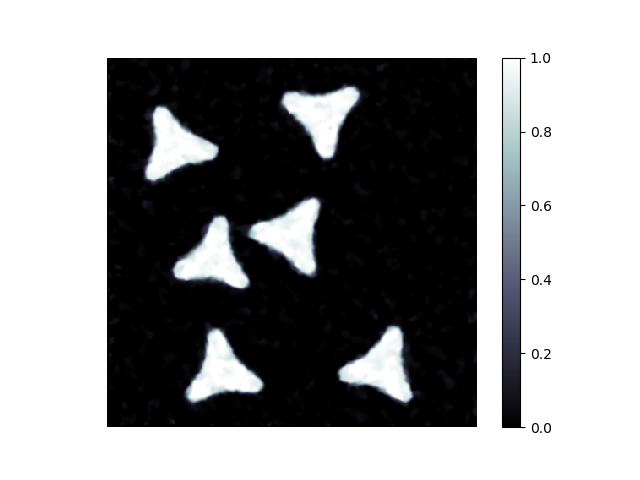}     
     \vskip-0.25\baselineskip
   \end{minipage}%
   \hfill
   \begin{minipage}[t]{0.2\textwidth}%
     \centering
    \includegraphics[trim=75 25 60 40, clip, width=\textwidth]{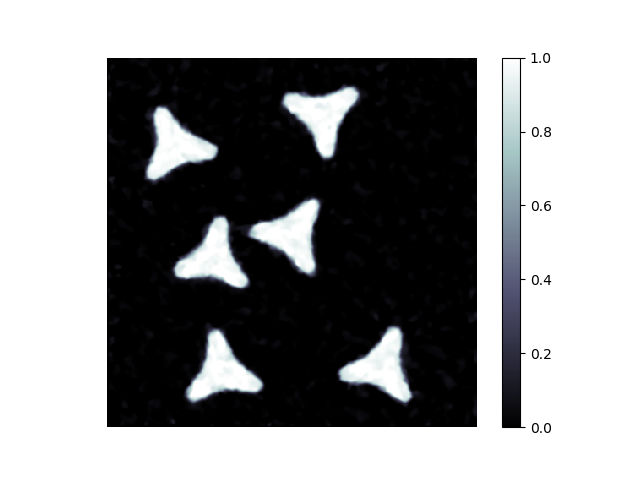}
     \vskip-0.25\baselineskip
   \end{minipage}%
\par\medskip      
\begin{minipage}[t]{0.2\textwidth}%
     \centering
     \includegraphics[trim=75 25 60 40, clip, width=\textwidth]{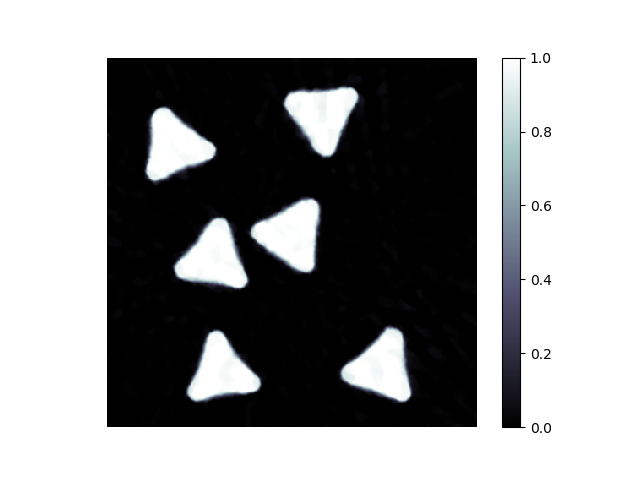}
     \vskip-0.25\baselineskip
   \end{minipage}%
   \hfill
   \begin{minipage}[t]{0.2\textwidth}%
     \centering
    \includegraphics[trim=75 25 60 40, clip, width=\textwidth]{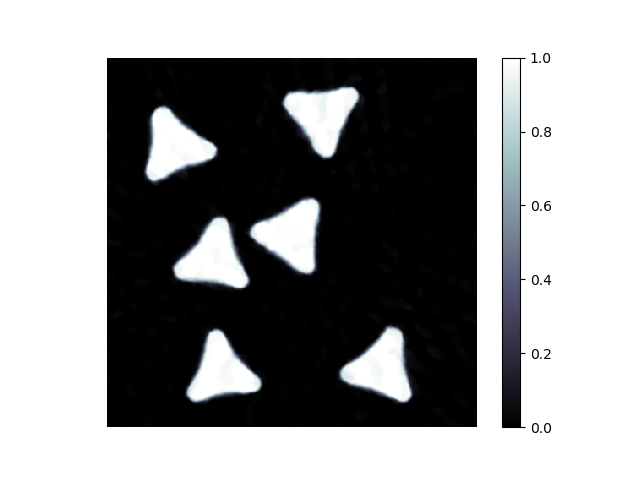}
     \vskip-0.25\baselineskip
   \end{minipage}%
   \hfill
   \begin{minipage}[t]{0.2\textwidth}%
     \centering
     \includegraphics[trim=75 25 60 40, clip, width=\textwidth]{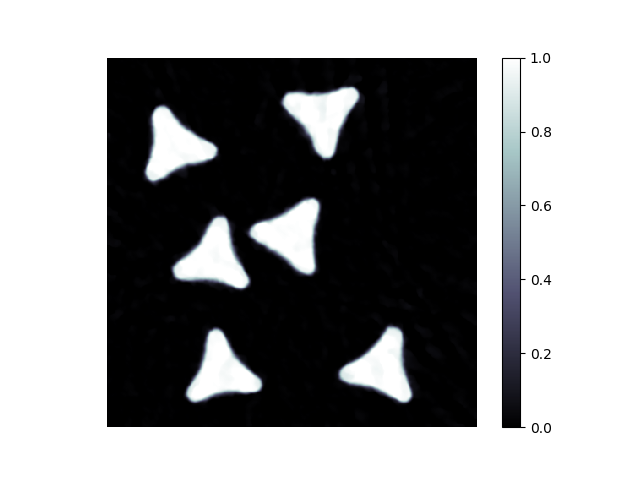}
     \vskip-0.25\baselineskip
   \end{minipage}%
      \hfill
   \begin{minipage}[t]{0.2\textwidth}%
     \centering
     \includegraphics[trim=75 25 60 40, clip, width=\textwidth]{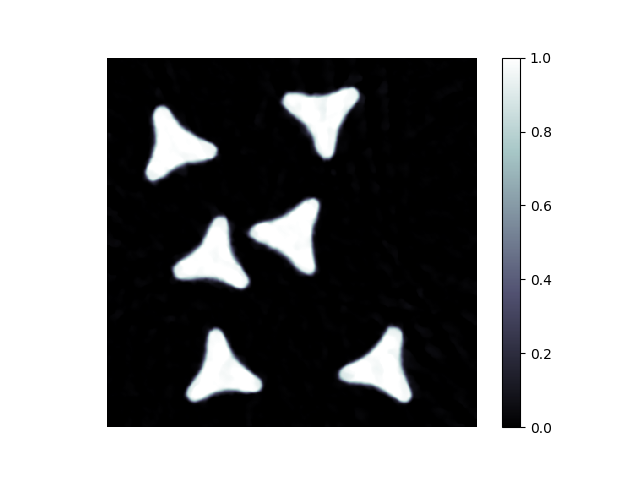}     
     \vskip-0.25\baselineskip
   \end{minipage}%
   \hfill
   \begin{minipage}[t]{0.2\textwidth}%
     \centering
    \includegraphics[trim=75 25 60 40, clip, width=\textwidth]{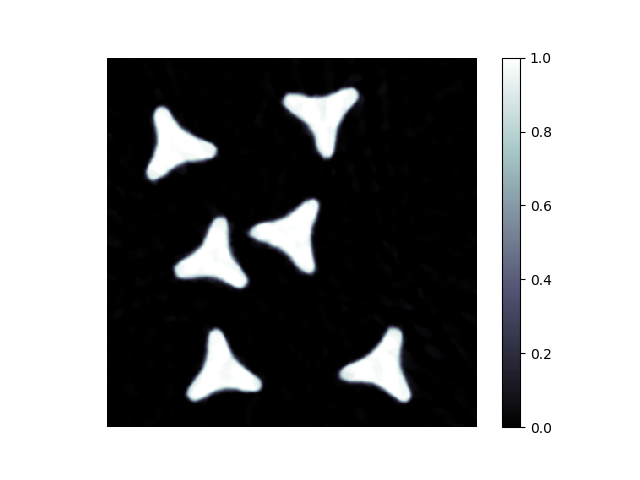}
     \vskip-0.25\baselineskip
   \end{minipage}%
\par\medskip      
   \begin{minipage}[t]{0.2\textwidth}%
     \centering
     \includegraphics[trim=75 25 60 40, clip, width=\textwidth]{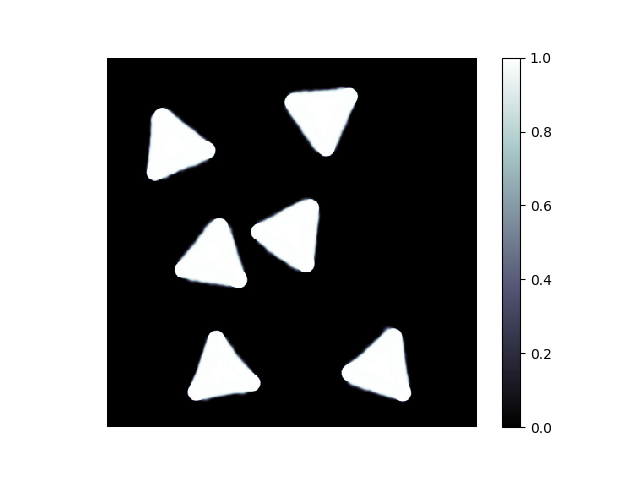}
     \vskip-0.25\baselineskip
     Gate 1
   \end{minipage}%
\hfill
   \begin{minipage}[t]{0.2\textwidth}%
     \centering
     \includegraphics[trim=75 25 60 40, clip, width=\textwidth]{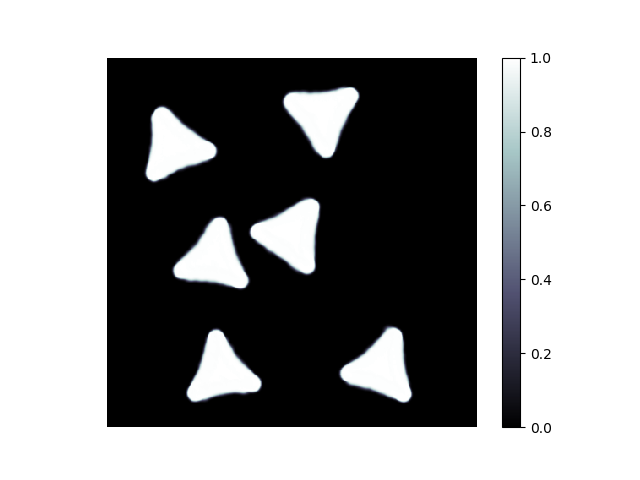}
     \vskip-0.25\baselineskip
     Gate 2
   \end{minipage}%
   \hfill
   \begin{minipage}[t]{0.2\textwidth}%
     \centering
    \includegraphics[trim=75 25 60 40, clip, width=\textwidth]{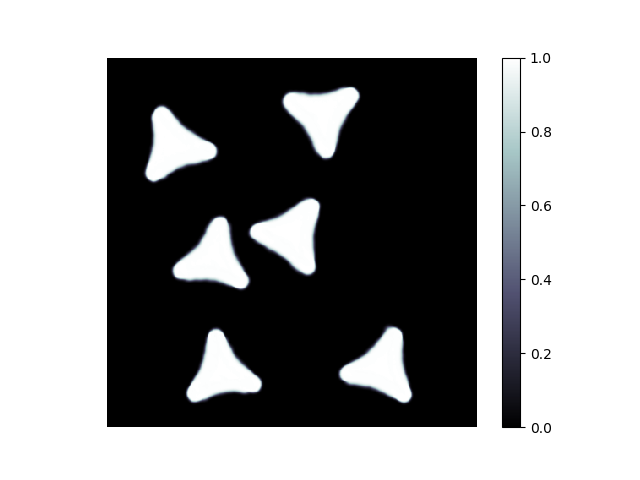}
     \vskip-0.25\baselineskip
          Gate 3
   \end{minipage}%
   \hfill
   \begin{minipage}[t]{0.2\textwidth}%
     \centering
     \includegraphics[trim=75 25 60 40, clip, width=\textwidth]{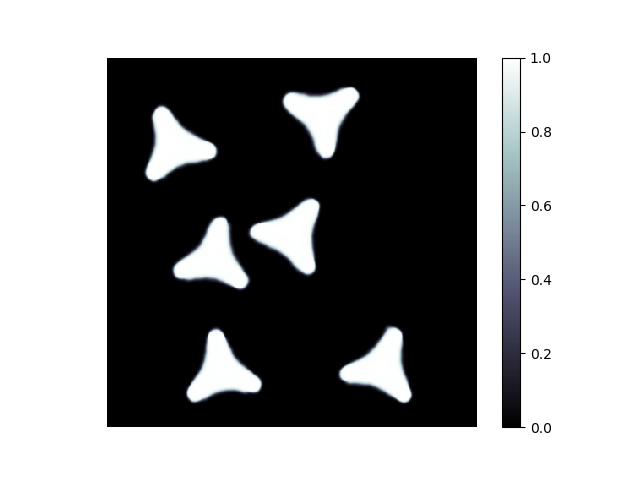}
     \vskip-0.25\baselineskip
     Gate 4
   \end{minipage}%
   \hfill
   \begin{minipage}[t]{0.2\textwidth}%
     \centering
     \includegraphics[trim=75 25 60 40, clip, width=\textwidth]{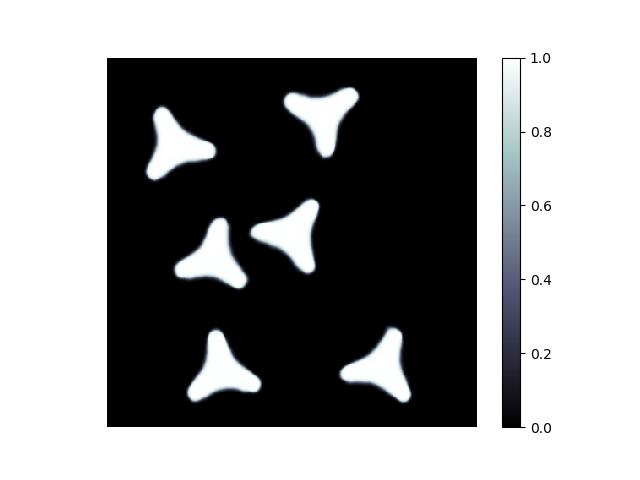}
     \vskip-0.25\baselineskip
     Gate 5
   \end{minipage}%
\caption{Multi-object phantom. Columns represent the gates and the first three rows are reconstructed spatiotemporal images for the data with noise levels $4.71$dB, $7.7$dB, and $14.67$dB, respectively. The last row shows the ground truth for each gate.}
\label{Test_suite_1:multi_object_phantom}
\end{figure}

Apart from the visual perception, the reconstruction is quantitatively compared 
using \ac{SSIM} and \ac{PSNR}, which is frequently used to 
evaluate image quality \cite{WaBoShSi04}.  
The \ac{SSIM} and \ac{PSNR} values are tabulated in \cref{Test_suite_1:different_data}. 
As listed in the above table, the corresponding \ac{SSIM} and \ac{PSNR} values are depended on 
\ac{SNR} of the data. The higher \ac{SNR}, the larger values of \ac{SSIM} and \ac{PSNR}. 
\begin{table}
\centering
\begin{tabular}{c | r r r r r}
&  \multicolumn{1}{c}{Gate 1} &  \multicolumn{1}{c}{Gate 2} & \multicolumn{1}{c}{Gate 3} 
 &  \multicolumn{1}{c}{Gate 4}  &  \multicolumn{1}{c}{Gate 5} \\ 
\hline                                   
 \multirow{2}{*}{Row 1}  & 0.4069          &  0.4208   &  0.4273   &  0.4305  & 0.4337  \\
     &22.10           &  23.02    &  23.27    &  23.40   & 23.64 \\
\hline                                   
 \multirow{2}{*}{Row 2}  &0.5934          &  0.6086   &  0.6131   &  0.6149  & 0.6156   \\
    & 25.36           &  27.22    &  27.37    &  27.66  &  27.86  \\
\hline
 \multirow{2}{*}{Row 3}    & 0.8411   &  0.8523   &  0.8564   &  0.8576   & 0.8587   \\  
     & 28.30    &  31.49    &  32.48    &  32.65  & 32.76   \\
     \hline
\end{tabular}
\caption{\Ac{SSIM} and \ac{PSNR} values of reconstructed spatiotemporal images compared to the related ground truths for the measured data with varying noise levels, see \cref{Test_suite_1:multi_object_phantom} for detailed images. Each table entry has two values, where the upper is the value of \ac{SSIM} and the bottom is the value of \ac{PSNR}, which correspond to the image on the counterpart position of row 1--row3 of \cref{Test_suite_1:multi_object_phantom}.
}
\label{Test_suite_1:different_data}
\end{table}

\paragraph{Comparison against static \ac{TV}-regularization}  
It is well-known that tomographic reconstruction by \ac{TV}-regularization outperforms other methods, 
such as \ac{FBP}, when the gradient of the image is sparse.
This is furthermore especially notable when data is under-sampled. 
In our tests we use a phantom (ground truth image) that has sparse gradient, 
so comparing against static \ac{TV}-regularization 
pitches our approach against one of the best static reconstruction methods.

For static \ac{TV}-regularization we disregard any temporal evolution, 
which is equivalent to simplify the spatiotemporal problem into one with a single gate.
The whole tomographic data set will then have 60 projection views. 
The regularization parameter for static \ac{TV}-regularization is selected depending 
on the \ac{SNR} of data in the same way as for spatiotemporal reconstruction. 
\begin{figure}[htbp]
\centering
\begin{minipage}[t]{0.33\textwidth}%
     \centering
     \includegraphics[trim=75 25 60 40, clip, width=\textwidth]{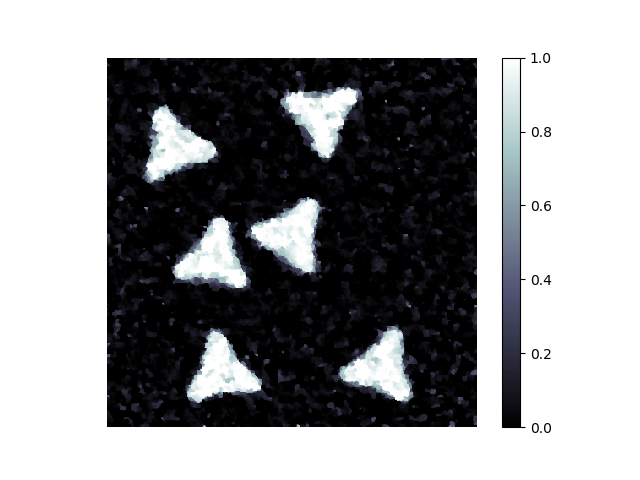}
     \vskip-0.25\baselineskip
   \end{minipage}%
   \hfill
   \begin{minipage}[t]{0.33\textwidth}%
     \centering
    \includegraphics[trim=75 25 60 40, clip, width=\textwidth]{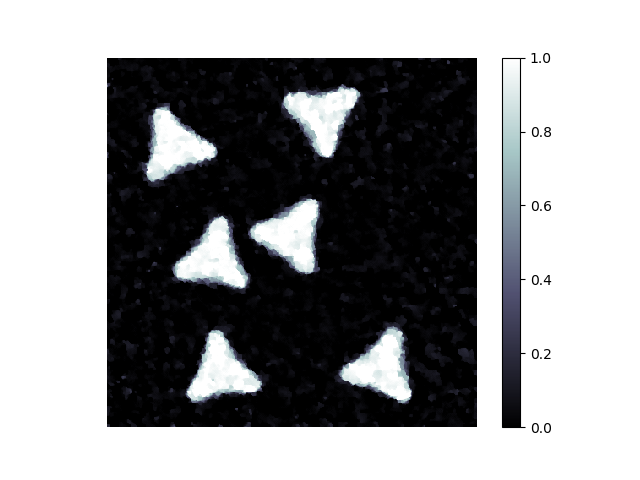}
     \vskip-0.25\baselineskip
   \end{minipage}%
   \hfill
   \begin{minipage}[t]{0.33\textwidth}%
     \centering
     \includegraphics[trim=75 25 60 40, clip, width=\textwidth]{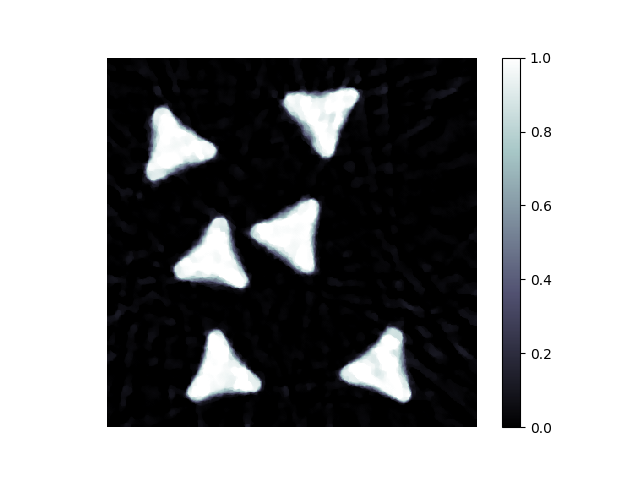}
     \vskip-0.25\baselineskip
   \end{minipage}%
\caption{\Ac{TV}-regularized reconstructions for the measured data with different noise levels $4.71$dB (left), $7.7$dB (middle), and $14.67$dB (right), respectively.}
\label{Test_suite_1:TV_reconstruction}
\end{figure}

Reconstructions obtained by static \ac{TV}-regularization are shown in \cref{Test_suite_1:TV_reconstruction}, 
the edges of which are become blurring against those by our method. In addition, 
the corresponding \ac{SSIM} and \ac{PSNR} values are listed in \cref{Test_suite_1:tv_different_data}. 
Compared \cref{Test_suite_1:different_data} with \cref{Test_suite_1:tv_different_data}, 
the values of \ac{SSIM} and \ac{PSNR} for static \ac{TV}-regularization is lower than those with the proposed method.
\begin{table}
\centering
\begin{tabular}{l | r r r r r}
&  \multicolumn{1}{c}{Gate 1} &  \multicolumn{1}{c}{Gate 2} & \multicolumn{1}{c}{Gate 3} 
 &  \multicolumn{1}{c}{Gate 4}  &  \multicolumn{1}{c}{Gate 5} \\ 
\hline                                   
 \multirow{2}{*}{Left}  & 0.3012          &  0.3163   &  0.3202   &  0.3146  & 0.3030  \\
     &18.57           &  19.94    &  20.42    &  19.98   & 18.80 \\
\hline                                   
 \multirow{2}{*}{Middle}  &0.4673          &  0.4867   &  0.4910   &  0.4840  & 0.4694   \\
    & 20.44           &  22.82    &  23.76    &  22.90  &  20.79  \\
\hline
 \multirow{2}{*}{Right}    & 0.6004   &  0.6239   &  0.6291   &  0.6212   & 0.6029   \\  
     & 21.42    &  24.71    &  26.40    &  25.00  & 21.95   \\
     \hline
\end{tabular}
\caption{\Ac{SSIM} and \ac{PSNR} values of reconstructed images by \ac{TV} method compared to each ground truth from Gate 1 to Gate 5 with the measured data with varying noise levels, see \cref{Test_suite_1:TV_reconstruction} for detailed images. Each entry has two values, where the upper is the value of \ac{SSIM} and the bottom is the value of \ac{PSNR}.}
\label{Test_suite_1:tv_different_data}
\end{table}

\subsubsection{Test suite 2: Sensitivity against selections of regularization parameters}

To solve the proposed model, three regularization parameters $\mu_1$, $\mu_2$ and $\sigma$ need to be selected. 
Hence the sensitivity test should be concerned against the selections of these parameters.  

As shown in the last row of \cref{Test_suite_2:heart_phantom}, a heart phantom with four gates (\ie $N = 4$) is used in this test, 
which is originated from \cite{GrMi07}. The image from each gate is consisting of a 
heart-like object with grey-values in $[0, 1]$, which is digitized using $120 \times 120$ pixels. 
The images of all gates are supported on a fixed rectangular 
domain $\Omega = [-4.5, 4.5] \times [-4.5, 4.5]$. For image in each gate, the noise-free data per view 
is measured by evaluating the 2D parallel beam scanning geometry with uniform $170$ bins, 
which is supported on the range of $[-6.4, 6.4]$. Then the additive Gaussian white noise is added onto the noise-free data. 
The resulting \ac{SNR} is about $14.9$dB. For gate $i\,(1 \le i \le N)$, the scanning views are distributed 
on $[(i-1)\pi /5, \pi + (i-1) \pi/5]$ evenly, which totally has 20 views. The factor of discretized time degree is $M = 2$. 
The gradient stepsizes are set as $\alpha = 0.01$ and $\beta = 0.05$, respectively.

We first employ \cref{algo:GDSB_4DCT} to gain an initial template after 50 iterations, 
then use \cref{algo:Alternating_reconstruction} to solve the proposed model. 
With selecting different values for regularization parameters, after 200 iterations, 
the reconstructed results are obtained, as shown in \cref{Test_suite_2:heart_phantom}. 
The detailed selections of varying parameter values can be referred to the caption of \cref{Test_suite_2:heart_phantom}.
For comparison, we also present the result for static \ac{TV}-regularization in \cref{Test_suite_2:TV_reconstruction} as did in the first 
test. As shown in \cref{Test_suite_2:heart_phantom}, 
the related reconstructed results are almost the same and close to the counterpart ground truth. 
However, the reconstructed result by static \ac{TV}-regularization in \cref{Test_suite_2:TV_reconstruction} 
is severely degraded.
\begin{figure}[htbp]
\centering
   \begin{minipage}[t]{0.2\textwidth}%
     \centering
     \includegraphics[trim=75 25 60 40, clip, width=\textwidth]{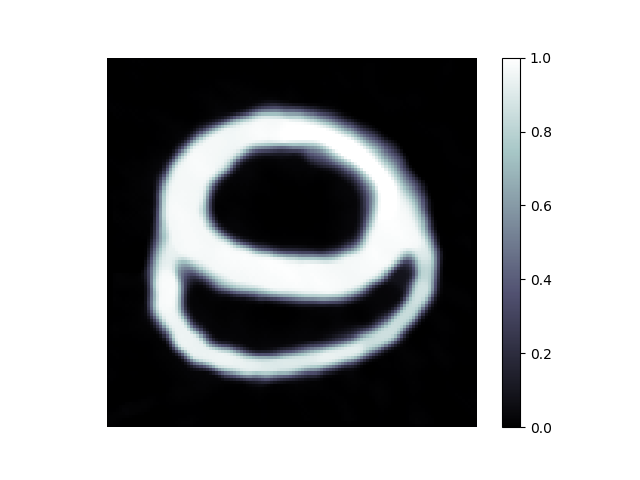}
     \vskip-0.25\baselineskip
   \end{minipage}%
   \hfill
   \begin{minipage}[t]{0.2\textwidth}%
     \centering
    \includegraphics[trim=75 25 60 40, clip, width=\textwidth]{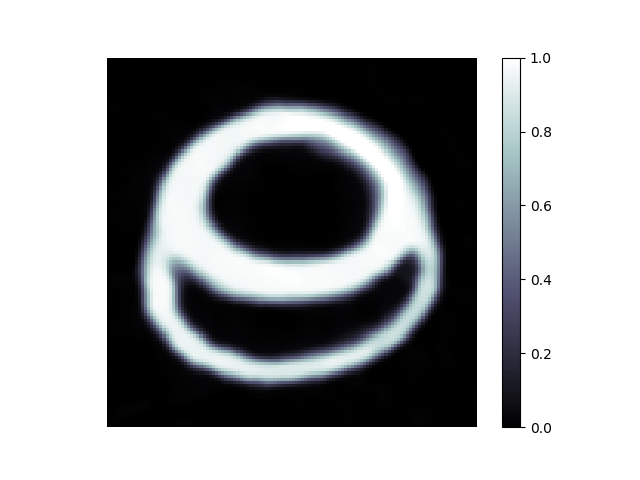}
     \vskip-0.25\baselineskip
   \end{minipage}%
   \hfill
   \begin{minipage}[t]{0.2\textwidth}%
     \centering
     \includegraphics[trim=75 25 60 40, clip, width=\textwidth]{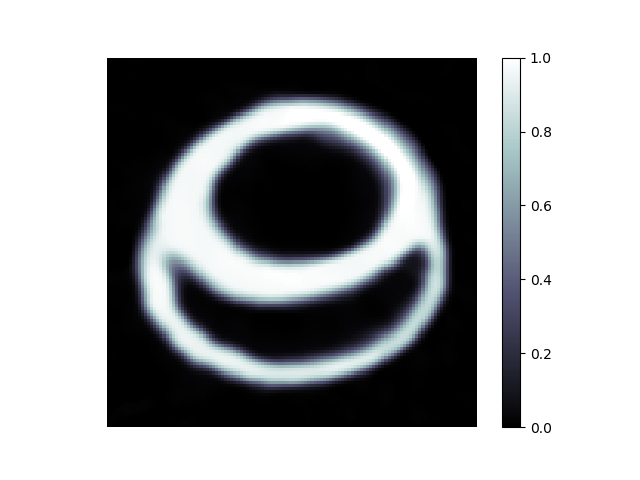}
     \vskip-0.25\baselineskip
   \end{minipage}%
      \hfill
   \begin{minipage}[t]{0.2\textwidth}%
     \centering
     \includegraphics[trim=75 25 60 40, clip, width=\textwidth]{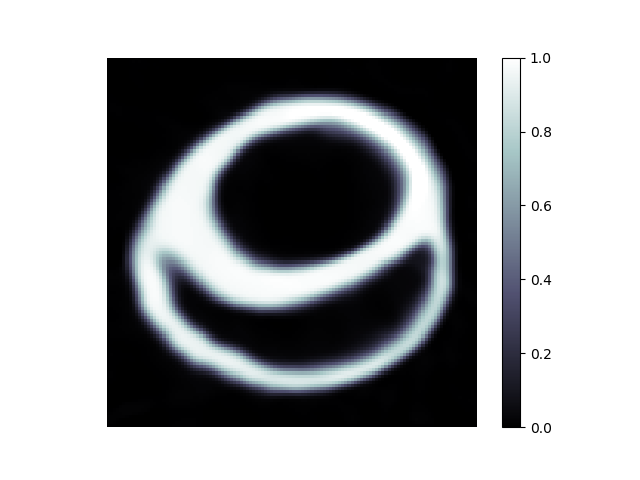}     
     \vskip-0.25\baselineskip
   \end{minipage}%
\par\medskip      
   \begin{minipage}[t]{0.2\textwidth}%
     \centering
     \includegraphics[trim=75 25 60 40, clip, width=\textwidth]{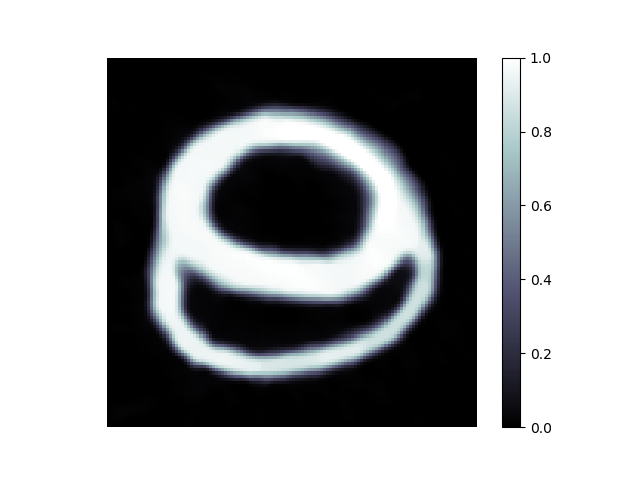}
     \vskip-0.25\baselineskip
   \end{minipage}%
   \hfill
   \begin{minipage}[t]{0.2\textwidth}%
     \centering
    \includegraphics[trim=75 25 60 40, clip, width=\textwidth]{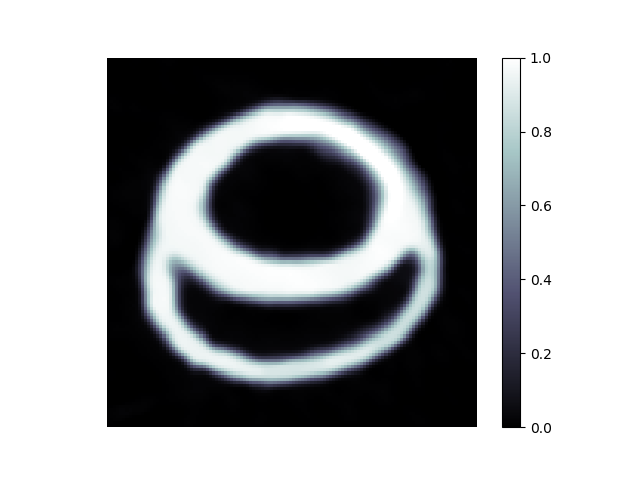}
     \vskip-0.25\baselineskip
   \end{minipage}%
   \hfill
   \begin{minipage}[t]{0.2\textwidth}%
     \centering
     \includegraphics[trim=75 25 60 40, clip, width=\textwidth]{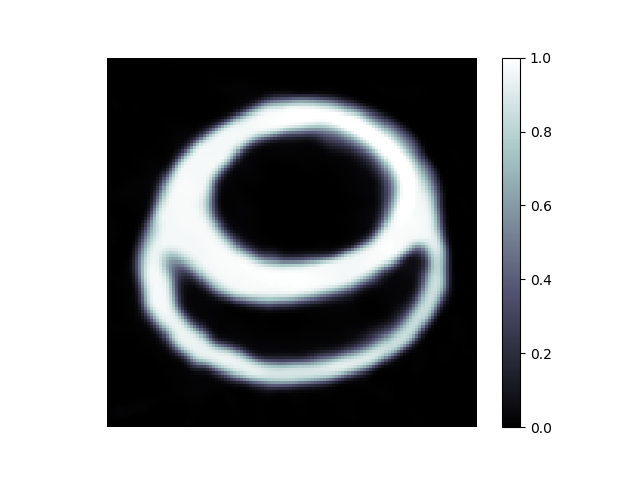}
     \vskip-0.25\baselineskip
   \end{minipage}%
      \hfill
   \begin{minipage}[t]{0.2\textwidth}%
     \centering
     \includegraphics[trim=75 25 60 40, clip, width=\textwidth]{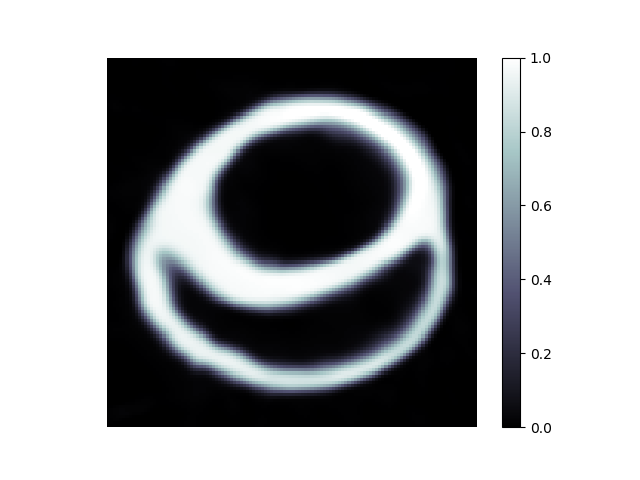}     
     \vskip-0.25\baselineskip
   \end{minipage}%
\par\medskip            
\begin{minipage}[t]{0.2\textwidth}%
     \centering
     \includegraphics[trim=75 25 60 40, clip, width=\textwidth]{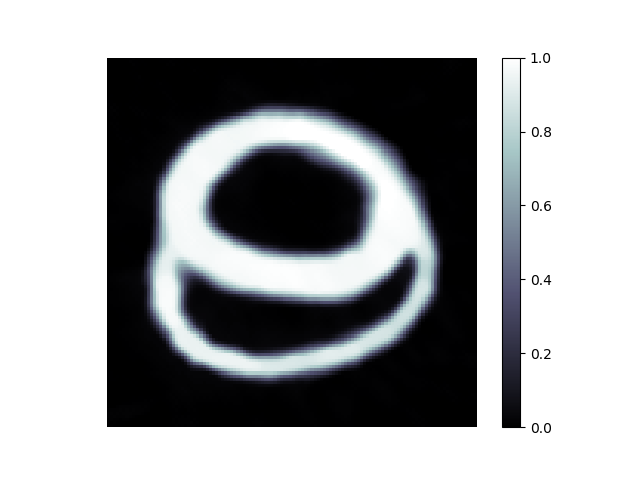}
     \vskip-0.25\baselineskip
   \end{minipage}%
   \hfill
   \begin{minipage}[t]{0.2\textwidth}%
     \centering
    \includegraphics[trim=75 25 60 40, clip, width=\textwidth]{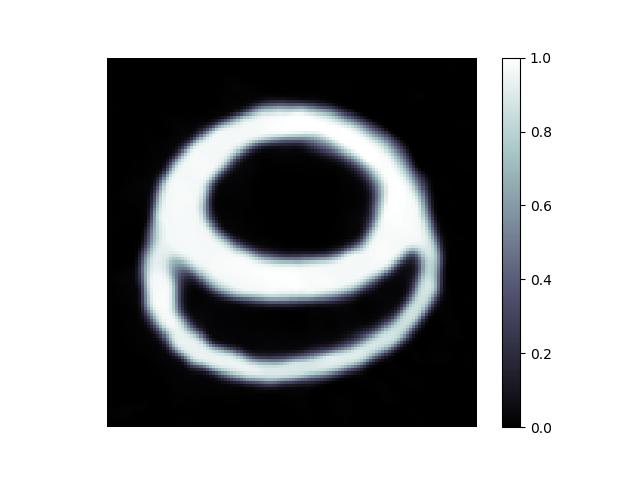}
     \vskip-0.25\baselineskip
   \end{minipage}%
   \hfill
   \begin{minipage}[t]{0.2\textwidth}%
     \centering
     \includegraphics[trim=75 25 60 40, clip, width=\textwidth]{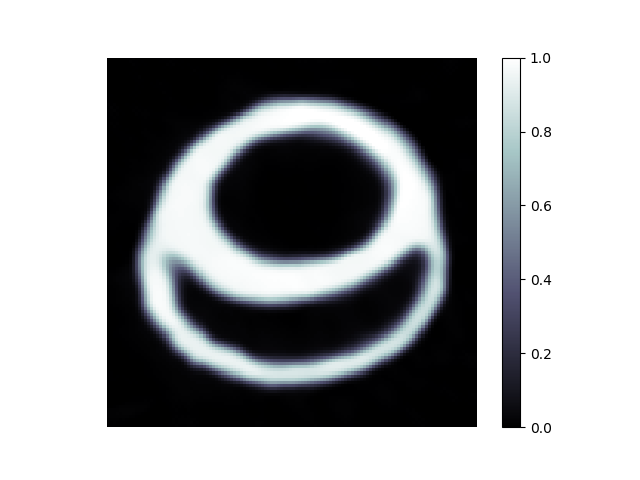}
     \vskip-0.25\baselineskip
   \end{minipage}%
      \hfill
   \begin{minipage}[t]{0.2\textwidth}%
     \centering
     \includegraphics[trim=75 25 60 40, clip, width=\textwidth]{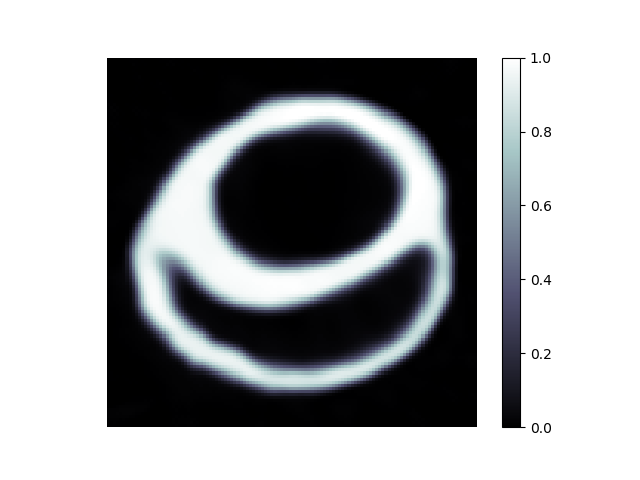}     
     \vskip-0.25\baselineskip
   \end{minipage}%
\par\medskip      
   \begin{minipage}[t]{0.2\textwidth}%
     \centering
     \includegraphics[trim=75 25 60 40, clip, width=\textwidth]{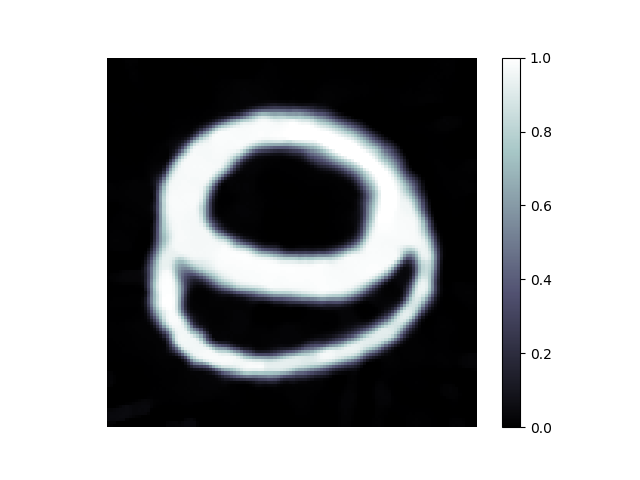}
     \vskip-0.25\baselineskip
   \end{minipage}%
   \hfill
   \begin{minipage}[t]{0.2\textwidth}%
     \centering
    \includegraphics[trim=75 25 60 40, clip, width=\textwidth]{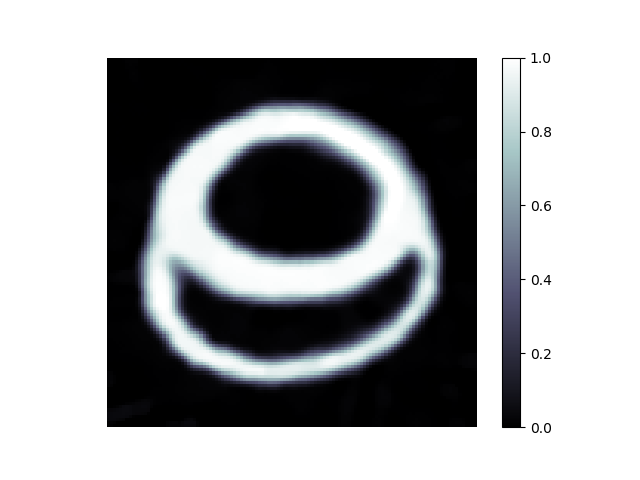}
     \vskip-0.25\baselineskip
   \end{minipage}%
   \hfill
   \begin{minipage}[t]{0.2\textwidth}%
     \centering
     \includegraphics[trim=75 25 60 40, clip, width=\textwidth]{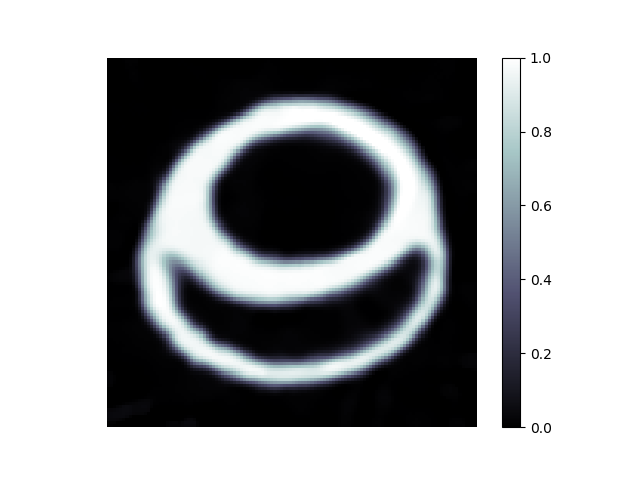}
     \vskip-0.25\baselineskip
   \end{minipage}%
      \hfill
   \begin{minipage}[t]{0.2\textwidth}%
     \centering
     \includegraphics[trim=75 25 60 40, clip, width=\textwidth]{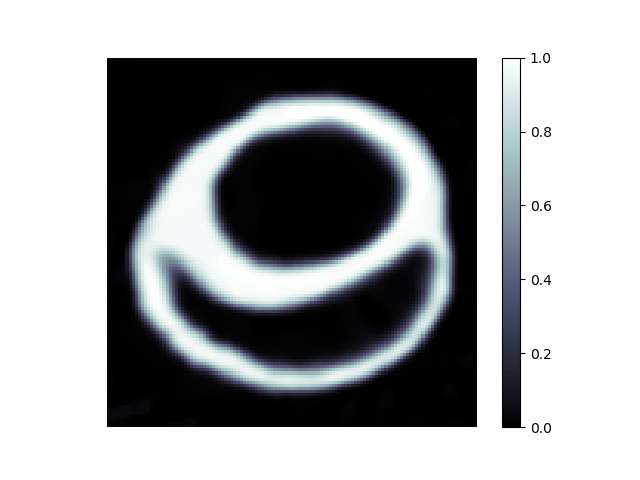}     
     \vskip-0.25\baselineskip
   \end{minipage}%
\par\medskip      
 \begin{minipage}[t]{0.2\textwidth}%
     \centering
     \includegraphics[trim=75 25 60 40, clip, width=\textwidth]{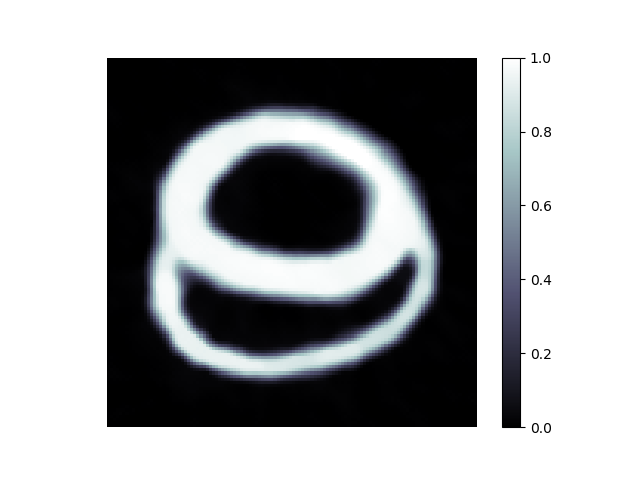}
     \vskip-0.25\baselineskip
   \end{minipage}%
   \hfill
   \begin{minipage}[t]{0.2\textwidth}%
     \centering
    \includegraphics[trim=75 25 60 40, clip, width=\textwidth]{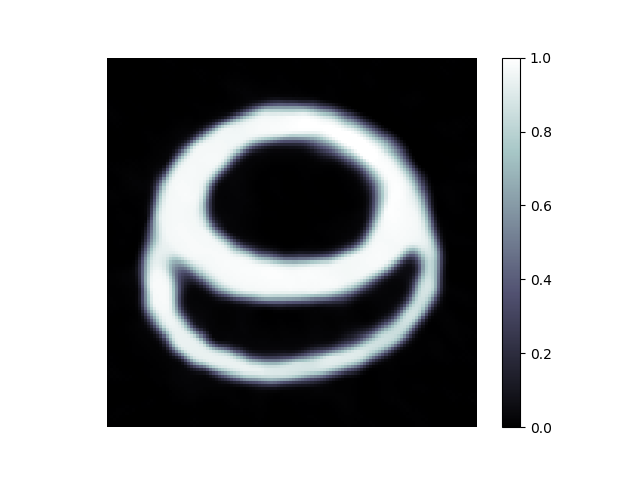}
     \vskip-0.25\baselineskip
   \end{minipage}%
   \hfill
   \begin{minipage}[t]{0.2\textwidth}%
     \centering
     \includegraphics[trim=75 25 60 40, clip, width=\textwidth]{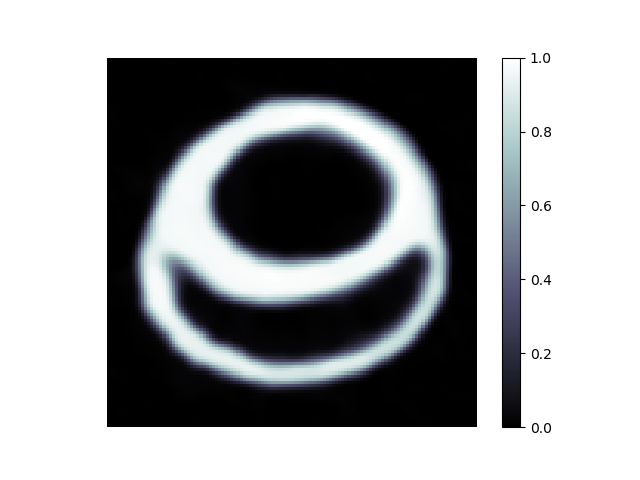}
     \vskip-0.25\baselineskip
   \end{minipage}%
      \hfill
   \begin{minipage}[t]{0.2\textwidth}%
     \centering
     \includegraphics[trim=75 25 60 40, clip, width=\textwidth]{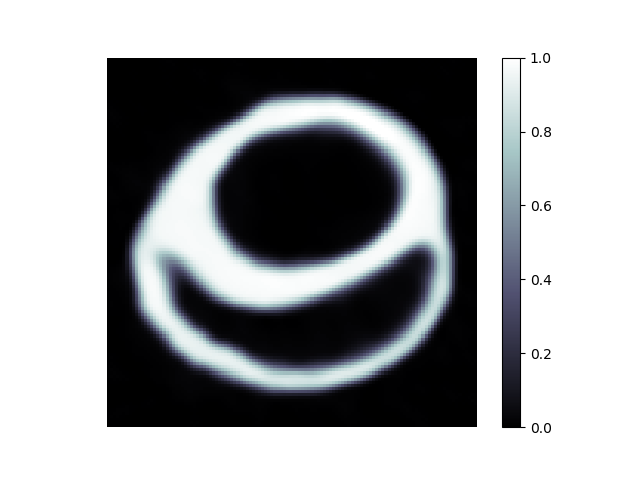}     
     \vskip-0.25\baselineskip
   \end{minipage}%
\par\medskip      
\begin{minipage}[t]{0.2\textwidth}%
     \centering
     \includegraphics[trim=75 25 60 40, clip, width=\textwidth]{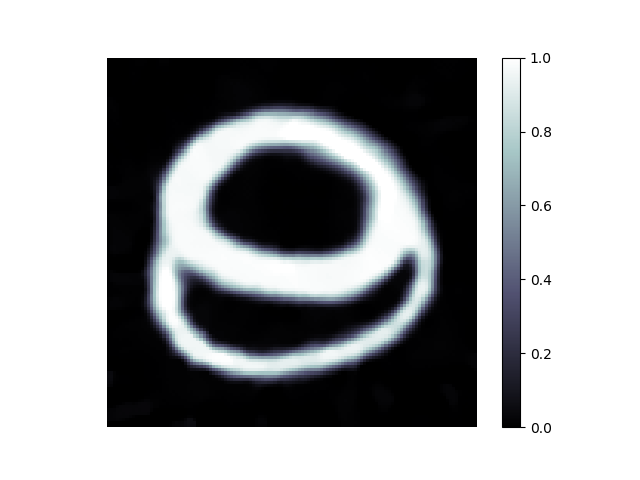}
     \vskip-0.25\baselineskip
   \end{minipage}%
   \hfill
   \begin{minipage}[t]{0.2\textwidth}%
     \centering
    \includegraphics[trim=75 25 60 40, clip, width=\textwidth]{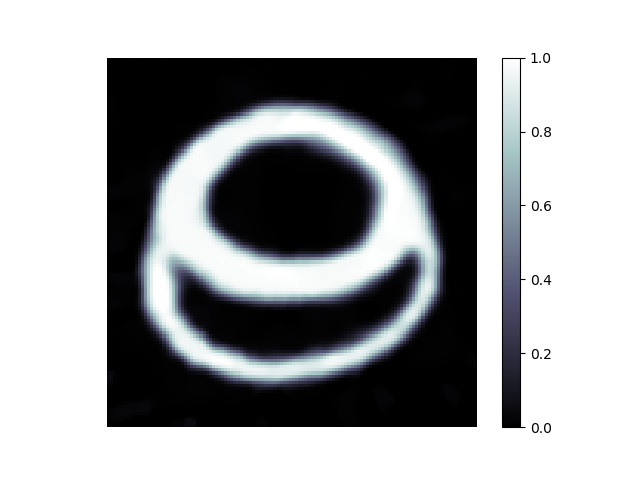}
     \vskip-0.25\baselineskip
   \end{minipage}%
   \hfill
   \begin{minipage}[t]{0.2\textwidth}%
     \centering
     \includegraphics[trim=75 25 60 40, clip, width=\textwidth]{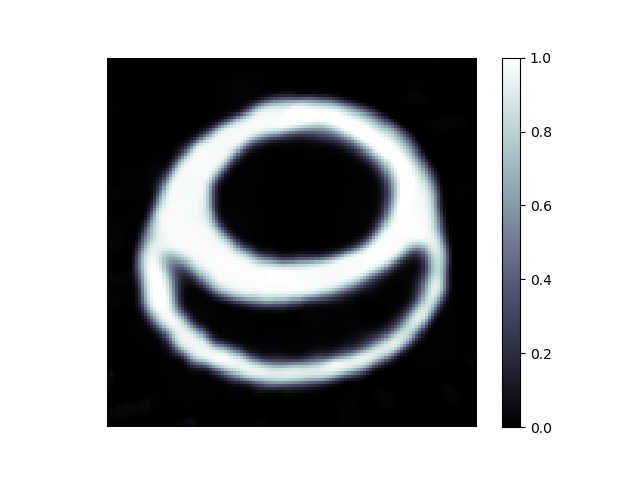}
     \vskip-0.25\baselineskip
   \end{minipage}%
      \hfill
   \begin{minipage}[t]{0.2\textwidth}%
     \centering
     \includegraphics[trim=75 25 60 40, clip, width=\textwidth]{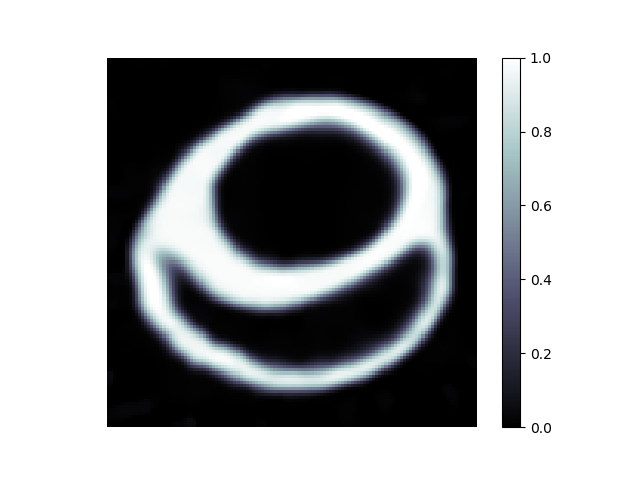}     
     \vskip-0.25\baselineskip
   \end{minipage}%
\par\medskip      
   \begin{minipage}[t]{0.2\textwidth}%
     \centering
     \includegraphics[trim=75 25 60 40, clip, width=\textwidth]{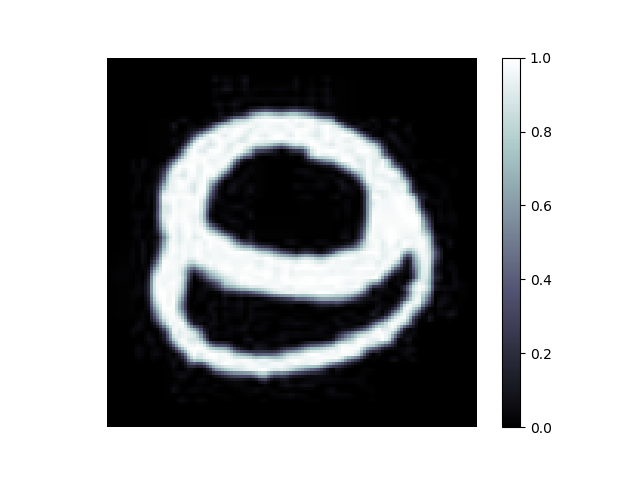}
     \vskip-0.25\baselineskip
     Gate 1
   \end{minipage}%
   \hfill
   \begin{minipage}[t]{0.2\textwidth}%
     \centering
    \includegraphics[trim=75 25 60 40, clip, width=\textwidth]{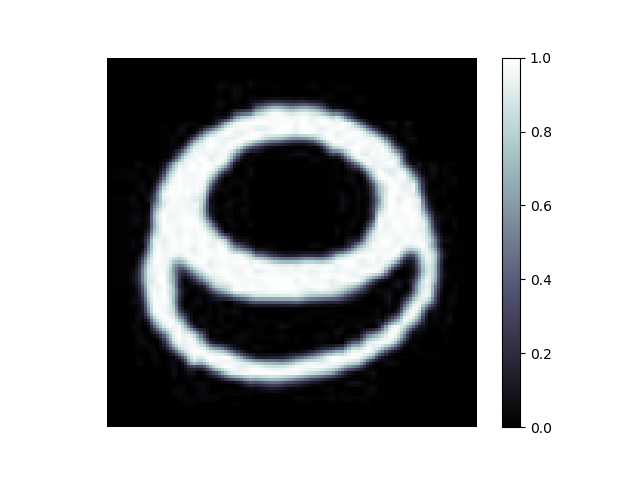}
     \vskip-0.25\baselineskip
     Gate 2
   \end{minipage}%
   \hfill
   \begin{minipage}[t]{0.2\textwidth}%
     \centering
     \includegraphics[trim=75 25 60 40, clip, width=\textwidth]{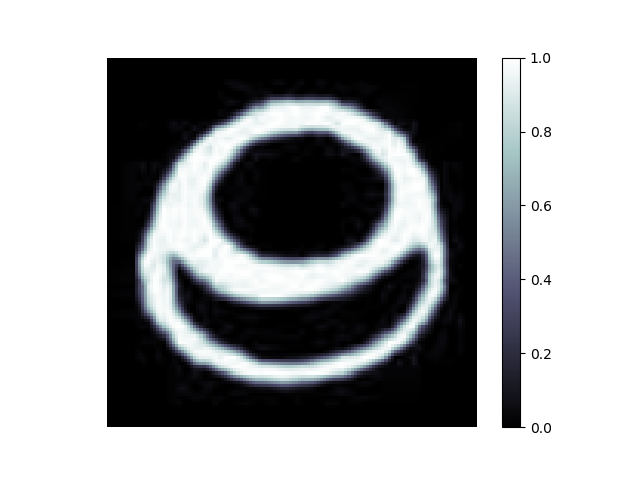}
     \vskip-0.25\baselineskip
     Gate 3
   \end{minipage}%
      \hfill
   \begin{minipage}[t]{0.2\textwidth}%
     \centering
     \includegraphics[trim=75 25 60 40, clip, width=\textwidth]{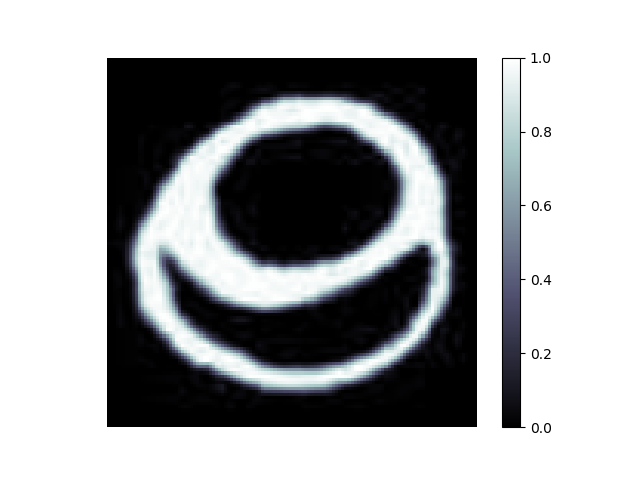}     
     \vskip-0.25\baselineskip
     Gate 4
   \end{minipage}%
\caption{Heart phantom. The columns are the 4 gates and the first 6 rows are reconstructed spatiotemporal images with parameter pairs ($\mu_1$, $\mu_2$, $\sigma$) chosen as $(0.01, 10^{-7}, 1.0)$, $(0.01, 10^{-6}, 1.0)$, $(0.01, 10^{-7}, 0.5)$, $(0.005, 10^{-7}, 0.5)$, $(0.01, 10^{-6}, 0.5)$, and $(0.005, 10^{-6}, 0.5)$. The last row shows the ground truth for each gate.}
\label{Test_suite_2:heart_phantom}
\end{figure}

\begin{figure}[htbp]
\centering
     \includegraphics[trim=75 25 60 40, clip, width=0.33\textwidth]{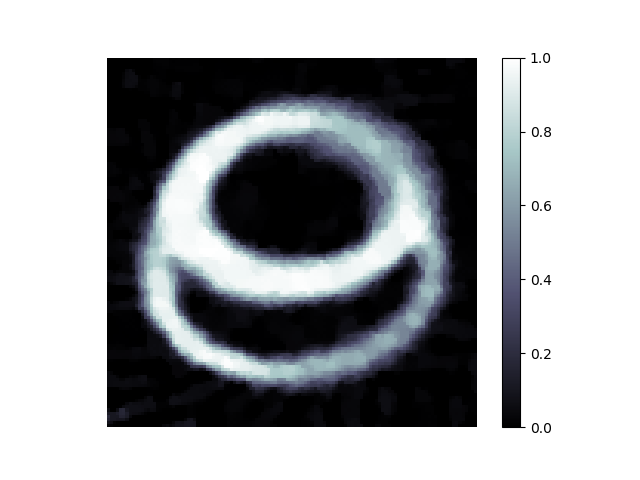}
\caption{Reconstruction by static \ac{TV}-regulrization for the measured data with comparable noise level. The selected parameter pair ($\mu_1$, $\mu_2$, $\sigma$) as $(0.005, 10^{-7}, 0.5)$.}
\label{Test_suite_2:TV_reconstruction}
\end{figure}

Furthermore, the \ac{SSIM} and \ac{PSNR} values are calculated in \cref{Test_suite_2:Sensitivity_table}. 
As given in the above table, the corresponding SSIM and PSNR values of the proposed method are 
relatively larger than those obtained by static \ac{TV}-regulrization, but are quite similar for 
different parameter pairs. 
\begin{table}
\centering
\begin{tabular}{c | r r r r}
&  \multicolumn{1}{c}{Gate 1} &  \multicolumn{1}{c}{Gate 2} & \multicolumn{1}{c}{Gate 3} 
 &  \multicolumn{1}{c}{Gate 4}   \\ 
\hline                                   
 \multirow{2}{*}{Row 1} & 0.8928          &  0.9382   &  0.9340   &  0.9235  \\
   & 24.25           &  28.44    &  27.64    &  26.28   \\
\hline                                   
 \multirow{2}{*}{Row 2}  &  0.8960          &  0.9415   &  0.9346   &  0.9234   \\
   &  24.30           &  28.47    &  27.67    &  26.37   \\
\hline
\multirow{2}{*}{Row 3}  &    0.9103   &  0.9497   &  0.9459   &  0.9343    \\  
  &   25.33    &  29.41    &  28.97    &  27.78   \\
\hline                                  
\multirow{2}{*}{Row 4}  &   0.8939          &  0.9370   &  0.9357   &  0.9297    \\  
   & 25.24           &  29.06    &  28.60    &  27.77   \\
\hline                                   
\multirow{2}{*}{Row 5} &  0.9087          &  0.9472   &  0.9462   &  0.9336   \\ 
 & 25.14           &  29.30    &  28.83    &  27.68  \\
 \hline                                 
\multirow{2}{*}{Row 6} &   0.8884          &  0.9339   &  0.9358   &  0.9295   \\ 
 &   25.23           &  29.06    &  28.65    &  27.74    \\
     \hline
 \multirow{2}{*}{TV} &   0.5641 & 0.7310  &  0.7458 &  0.5969   \\ 
 &   14.09 &  19.09 & 18.96 & 14.01    \\
     \hline
\end{tabular}
\caption{Row 1--row 6: \ac{SSIM} and \ac{PSNR} values of reconstructed spatiotemporal images compared to the related ground truths for varying values of the regularization parameters $\mu_1$, $\mu_2$, and the kernel width $\sigma$, see \cref{Test_suite_2:heart_phantom} for detailed images. Each entry has two values, where the upper is value of \ac{SSIM} and the bottom is  value of \ac{PSNR}, which corresponds to the image on the counterpart position of row 1--row 6 of \cref{Test_suite_2:heart_phantom}. The last row:  \Ac{SSIM} and \ac{PSNR} values of reconstructed image by \ac{TV}-based method compared to each ground truth from gates 1 to 4 by the measured data with comparable noise level.}
\label{Test_suite_2:Sensitivity_table}
\end{table}

As shown in \cref{Test_suite_2:Sensitivity_table}, these values are a little bit decreased when the 
value of kernel parameter $\sigma$ is changed from  $1.0$ to $0.5$ with fixed $\mu_1$ and $\mu_2$, 
as compared the values between row 1 and row 3, also row 2 and row 4, for instance. 
Therefore, this test demonstrates that to some extent the proposed method is not sensitive to 
the precise selection of the regularization parameters under the visual perception and 
the quantitative comparison (\ac{SSIM} and \ac{PSNR}). However, those values are 
selected too big or too small, which would causes over- or under-regularized results.

\section{Conclusions and the future work}\label{sec:Conclusions}

A general framework of variational model has been investigated for joint image 
reconstruction and motion estimation in spatiotemporal imaging, which is based on 
the deformable templates from shape theory. Along this framework, we proposed a new 
variational model for solving the above joint problem using the principle of \ac{LDDMM}. 
The proposed model is equivalent to a \ac{PDE}-constrained optimal control problem. 
Based on the equivalency, we made a mathematical 
comparison against the joint \ac{TV}-\ac{TV} optical flow based model \cite{BuDiSch18}, 
which showed that our method can guarantee 
elastically diffeomorphic deformations, and is of benefit to 
the practical computation additionally. Furthermore, the theoretical comparison was also performed between the proposed 
model and other diffeomorphic motion models, which demonstrated that the optimal velocity field of our model 
is distributed \wrt time $t$ averagely, and nonvanishing neither on the initial nor on the end time point. 
We also presented an efficiently computational method for the time-discretized version of the 
proposed model, which showed that its optimal solution is consistent with that of the time-continuous one, 
but this is not the case for the diffeomorphic motion model in \cite{HiSzWaSaJo12}.

An alternately gradient descent algorithm was designed to solve the time-discretized proposed model, 
where the main calculations were only based on the easy-implemented linearized deformations. 
For spatiotemporal (2D space + time) parallel beam tomographic imaging, the computational cost 
of the algorithm is then $O(n^3NN_v)$ and its memory requirement scales as $O(n^2MN^2)$.  
With \cref{algo:Alternating_reconstruction}, we have evaluated the performance of the proposed model in dealing 
with the 2D space + time tomography in the case of very sparse and/or highly noisy data. As shown in these visual and 
quantitative results, the new method can yield reconstructed spatiotemporal images of high quality for the 
above difficult problems.

The future work will focus on the theoretical analysis of the proposed model, such as the existence and 
uniqueness of the solution, the convergence analysis of the proposed algorithm, and its extensions and 
applications to more complicated modalities in spatiotemporal imaging.

\appendix
\section{Optimality conditions}
\label{sec:Optimality_conditions}

The goal is to characterize optimality conditions for \cref{eq:VarReg_LDDMM_2}.  Let us begin with the following result.

\begin{lemma}[\cite{Yo10}]\label{lem:DerivDiffeo_flow}
Let $\velocityfield,\velocityfieldother \in \Xspace{2}$, and $\diffeo_{0,t}^{\velocityfield}$ denote the solution to 
the ODE in \cref{eq:FlowEq} with given $\velocityfield$ at time $t$, and $\gelement{s,t}{\velocityfield}$ be 
defined as in \cref{eq:FlowRelation}. Then, 
\begin{equation}
   \label{eq:PhiInvDeriv}
    \frac{d}{d \epsilon} \gelement{s,t}{\velocityfield + \epsilon\velocityfieldother} (x) \Bigl\vert_{\epsilon=0}
    = \int_{s}^{t} \Diff\bigl( \gelement{\tau,t}{\velocityfield} \bigr)\bigl( \gelement{s,\tau}{\velocityfield}(x) \bigr) 
             \Bigl( \velocityfieldother\bigl(\tau, \gelement{s,\tau}{\velocityfield}(x) \bigr) \Bigr) \dint \tau 
           \end{equation}
for $x \in \domain$ and $0 \leq s,t \leq 1$.
\end{lemma}

Using the result of \cref{lem:DerivDiffeo_flow}, we have the statement below. 

\begin{lemma}\label{lem:DiffEvolutionOperator_tt}
Let the assumptions in \cref{lem:DerivDiffeo_flow} hold and $\template \in \LpSpace^2(\domain,\Real)$ be differentiable. 
Considering the group action in \cref{eq:GeometricDeformation}, then 
\begin{equation}\label{eq:DiffEvolutionOperator_tt}
  \frac{d}{d \epsilon} \bigl(\diffeo_{0,t}^{\velocityfield +  \epsilon\velocityfieldother} . \template\bigr) (x)  \Bigl\vert_{\epsilon=0} 
  = - \int_{0}^{t} 
      \Bigl\langle 
      \nabla (\diffeo_{0,\tau}^{\velocityfield} . \template)\bigl( \gelement{t,\tau}{\velocityfield}(x)\bigr),
           \velocityfieldother\bigl(\tau,\gelement{t,\tau}{\velocityfield}(x) \bigr) 
      \Bigr\rangle_{\Real^n} \dint \tau                       
\end{equation}  
for $x \in \domain$. 
\end{lemma}
\begin{proof}
By chain rule we get  
\begin{equation}\label{eq:chainrule_evolution}
\frac{d}{d \epsilon} \bigl(\diffeo_{0,t}^{\velocityfield +  \epsilon\velocityfieldother} . \template\bigr)(x)  \Bigl\vert_{\epsilon=0}   =   \biggl\langle 
      \grad \template \bigl(\gelement{t,0}{\velocityfield}(x)\bigr),
      \frac{d}{d \epsilon} \gelement{t,0}{\velocityfield + \epsilon\velocityfieldother}(x)\Bigl\vert_{\epsilon=0}
      \biggr\rangle_{\Real^n}.
\end{equation}
Using \cref{lem:DerivDiffeo_flow}, we know 
 \begin{equation}\label{eq:derivative_inverse} 
 \frac{d}{d \epsilon} \gelement{t,0}{\velocityfield + \epsilon\velocityfieldother}(x)\Bigl\vert_{\epsilon=0} = 
 - \int_{0}^{t} \Diff\bigl( \gelement{\tau,0}{\velocityfield} \bigr)\bigl( \gelement{t,\tau}{\velocityfield}(x) \bigr) 
             \Bigl( \velocityfieldother\bigl(\tau, \gelement{t,\tau}{\velocityfield}(x) \bigr) \Bigr) \dint \tau.
 \end{equation}
 Inserting \cref{eq:derivative_inverse} into \cref{eq:chainrule_evolution}, we immediately prove  \cref{eq:DiffEvolutionOperator_tt}.
\end{proof}

The following result is a direct consequence of the above definition and \cref{lem:DiffEvolutionOperator_tt}.

\begin{lemma}\label{lem:data_matching_derivative}
Let the assumptions in \cref{lem:DiffEvolutionOperator_tt} hold and 
$\DataDisc_{g_t}\colon \RecSpace \to \Real$ be defined as \cref{eq:LDDMM_match_short}. 
Assuming that $\DataDisc_{g_t}$ is differentiable. Then
\begin{multline}\label{eq:Energy_func_deformation}
  \frac{d}{d \epsilon} \DataDisc_{g_t} \bigl(\diffeo_{0,t}^{\velocityfield + \epsilon\velocityfieldother} . \template \bigr)\Bigl\vert_{\epsilon=0} \\
    =  \int_{0}^{t}\Bigl\langle - \bigl\vert \Diff(\gelement{\tau,t}{\velocityfield}) \bigr\vert \grad\DataDisc_{g_t} \bigl(\diffeo_{0,t}^{\velocityfield} . \template\bigr)\bigl( \gelement{\tau,t}{\velocityfield}\bigr) \nabla (\diffeo_{0,\tau}^{\velocityfield} . \template),  \velocityfieldother(\tau, \Cdot) \Bigr\rangle_{\LpSpace^2(\domain,\Real^n)} \dint \tau,
  \end{multline}  
where $\grad\DataDisc_{g_t}$ denotes the gradient of $\DataDisc_{g_t}$.
\end{lemma}

We are now ready to characterize optimality conditions for \cref{eq:VarReg_LDDMM_2}.
\begin{theorem}\label{thm:energy_functional_derivative_2}
Let the assumptions in \cref{lem:data_matching_derivative} hold and 
$\GoalFunctionalV_1 \colon \RecSpace \times \Xspace{2} \to \Real$ denote the objective 
functional in \cref{eq:VarReg_LDDMM_2}, \ie
\begin{equation}\label{eq:LDDMM_vector_part2}
\GoalFunctionalV_1(\template, \velocityfield) :=  \int_{0}^{1} \left[\DataDisc_{\data_t}\bigl( \diffeo_{0,t}^{\velocityfield} . \template \bigr)  + 
   \mu_2\int_{0}^{t} \bigl\Vert \velocityfield(\tau,\cdot) \bigr\Vert^2_{\Vspace}\dint \tau \right] \dint t  + \mu_1 \RegFunc_1(\template).
\end{equation}
Assuming that the regularization term $\RegFunc_1$ is differentiable, and $\Vspace$ is a \ac{RKHS} with a reproducing 
kernel $\kernel \colon \domain \times \domain \to \Matrix_{+}^{n \times n}$. Then the $\Xspace{2}$--gradient \wrt the 
velocity field $\velocityfield$ of $\GoalFunctionalV_1(\template,\Cdot) \colon \Xspace{2} \to \Real$ is 
\begin{multline}\label{eq:Energy_functional_gradient_2}
  \grad_{\velocityfield}\GoalFunctionalV_1(\template, \velocityfield)(t,\Cdot)  
    = \Koperator\left(- \nabla (\diffeo_{0,t}^{\velocityfield} . \template) \int_{t}^{1} \bigl\vert \Diff(\gelement{t,\tau}{\velocityfield}) \bigr\vert \grad\DataDisc_{g_{\tau}} \bigl(\diffeo_{0,\tau}^{\velocityfield} . \template\bigr)\bigl( \gelement{t,\tau}{\velocityfield}\bigr)  \dint \tau \right) \\
    + 2\mu_2 (1-t)\velocityfield(t,\Cdot)
  \end{multline}
for $0\leq t \leq 1$ and where $\Koperator(\varphi) = \int_{\Omega} \kernel(\Cdot, y)\varphi(y) \dint y$.
Furthermore, the gradient \wrt the template $\template$ of $\GoalFunctionalV_1(\Cdot,\velocityfield) \colon \RecSpace \to \Real$  is   
\begin{equation}\label{eq:Energy_functional_gradient_template_2}
  \grad_{\template}\GoalFunctionalV_1(\template, \velocityfield)  \\
= \int_0^1 \bigl\vert \Diff(\gelement{0,t}{\velocityfield}) \bigr\vert \grad\DataDisc_{g_t} \bigl(\diffeo_{0,t}^{\velocityfield} . \template\bigr)\bigl( \gelement{0,t}{\velocityfield}\bigr) \dint t + \mu_1 \grad \RegFunc_1(\template), 
  \end{equation}
where $\grad\RegFunc_1$ denotes the gradient of $\RegFunc_1 \colon \RecSpace \to \Real$.
Finally, the optimality conditions for \cref{eq:VarReg_LDDMM_2} read as 
\begin{equation}\label{eq:OptCond_2}
\left\{
\begin{array}{ll}
 \grad_{\velocityfield}\GoalFunctionalV_1(\template, \velocityfield)(t,\Cdot) = 0, \\[0.5em]
 \grad_{\template}\GoalFunctionalV_1(\template, \velocityfield) = 0.
\end{array}
\right.
\end{equation}
\end{theorem}
\begin{proof}
Applying the result in \cref{lem:data_matching_derivative}, we immediately have  
\begin{multline*}
  \frac{d}{d \epsilon}  \GoalFunctionalV_1(\template, \velocityfield + \epsilon\velocityfieldother)  \Bigl\vert_{\epsilon=0} \\ 
    =  \int_{0}^{1} \int_{0}^{t}\Bigl\langle - \bigl\vert \Diff(\gelement{\tau,t}{\velocityfield}) \bigr\vert \grad\DataDisc_{g_t} \bigl(\diffeo_{0,t}^{\velocityfield} . \template\bigr)\bigl( \gelement{\tau,t}{\velocityfield}\bigr) \nabla (\diffeo_{0,\tau}^{\velocityfield} . \template),  \velocityfieldother(\tau, \Cdot) \Bigr\rangle_{\LpSpace^2(\domain,\Real^n)} \dint \tau \dint t \\
    + 2\mu_2 \int_{0}^{1}\int_{0}^{t}\bigl\langle \velocityfield(\tau,\Cdot),  \velocityfieldother(\tau, \Cdot)\bigr\rangle_{\Vspace}\dint \tau \dint t.
  \end{multline*}  
Changing the order of integration in the above equation gives
 \begin{multline}\label{eq:Energy_functional_derivative_1}
  \frac{d}{d \epsilon}  \GoalFunctionalV_1(\template, \velocityfield + \epsilon\velocityfieldother)  \Bigl\vert_{\epsilon=0}  \\
   =  \int_{0}^{1} \Bigl\langle -\nabla (\diffeo_{0,\tau}^{\velocityfield} . \template) \int_{\tau}^{1} \bigl\vert \Diff(\gelement{\tau,t}{\velocityfield}) \bigr\vert \grad\DataDisc_{g_t} \bigl(\diffeo_{0,t}^{\velocityfield} . \template\bigr)\bigl( \gelement{\tau,t}{\velocityfield}\bigr)  \dint t,  \velocityfieldother(\tau, \Cdot) \Bigr\rangle_{\LpSpace^2(\domain,\Real^n)}\dint \tau \\
    + 2\mu_2 \int_{0}^{1}\bigl\langle (1-\tau)\velocityfield(\tau,\Cdot),  \velocityfieldother(\tau, \Cdot)\bigr\rangle_{\Vspace} \dint \tau.
     \end{multline}  
As $\Vspace$ is a \ac{RKHS} with a reproducing kernel represented 
by $\kernel \colon \domain \times \domain \to \Matrix_{+}^{n \times n}$, then 
\begin{equation}\label{eq:RKHS_L2}
 \langle \vfield, \vfieldother \rangle_{\LpSpace^2(\domain,\Real^n)} 
   = \biggl\langle \int_{\domain} \kernel(\Cdot, y) \vfield(y) \dint y, \vfieldother \biggr\rangle_{\Vspace} \quad\text{for $ \vfield, \vfieldother \in \Vspace$}.
\end{equation}
Combining \cref{eq:Energy_functional_derivative_1} with \cref{eq:RKHS_L2} proves \cref{eq:Energy_functional_gradient_2}.
Finally, the results in \cref{eq:Energy_functional_gradient_template_2} and \cref{eq:OptCond_2} are rather 
straightforward to obtain, so we omit their proofs. 
\end{proof}

\section{First-order variation of $\GoalFunctionalV_ {\template}$}
\label{sec:Gradient_computation_GoalFunctionalV}

\begin{theorem}\label{thm:energy_functional_time_discretized_derivative}
Let the assumptions in \cref{lem:DiffEvolutionOperator_tt} hold. 
Suppose $\GoalFunctionalV_ {\template} \colon \Xspace{2} \to \Real$ is given 
as in \cref{eq:VarReg_LDDMM_deformation_time_discrete} and $\Vspace$ is a \ac{RKHS} with a reproducing kernel 
$\kernel \colon \domain \times \domain \to \Matrix_{+}^{n \times n}$. 
The $\Xspace{2}$--gradient of $\GoalFunctionalV_{\template}$ is 
\begin{multline}\label{eq:Energy_functional_time_discretized_gradient}
  \grad\GoalFunctionalV_ {\template}(\velocityfield)(t,x)  
    =   - \frac{2}{N}\int_{\domain}\kernel(x,y)\nabla (\template \circ \gelement{t,0}{\velocityfield})(y) \sum_{\{i\geq 1 : t_i \geq t\}}h_{t, t_i}^{\template,\velocityfield}(y) \dint y \\
    + \frac{2\mu_2}{N}\sum_{\{i\geq 1 : t_i \geq t\}}\velocityfield_{t, t_i}(x),
  \end{multline}
  for $0\leq t \leq 1$ and $x\in \domain$.
\end{theorem}
\begin{proof}
From \cref{lem:DiffEvolutionOperator_tt} it is not difficult to derive   
\begin{align*}
\frac{d}{d \epsilon}  \GoalFunctionalV_ {\template}(\velocityfield + \epsilon\velocityfieldother)  \Bigl\vert_{\epsilon=0} 
&=  \frac{1}{N}\sum_{i=1}^{N} \int_{0}^{t_i}\Bigl\langle -2 \eta_{\tau, t_i}^{\template, \velocityfield} \nabla (\template \circ \gelement{\tau,0}{\velocityfield}),  \velocityfieldother(\tau, \Cdot) \Bigr\rangle_{\LpSpace^2(\domain,\Real^n)} \dint \tau \\
  &\quad\quad  + \frac{\mu_2}{N}\sum_{i=1}^{N} \int_{0}^{t_i}\Bigl\langle 2\velocityfield(\tau,\Cdot),  \velocityfieldother(\tau, \Cdot)\Bigr\rangle_{\Vspace}\dint \tau \\
        &=   \int_{0}^{1} \biggl\langle -\frac{2}{N}\sum_{i=1}^{N}h_{\tau, t_i}^{\template,\velocityfield} \nabla (\template \circ \gelement{\tau,0}{\velocityfield}),  \velocityfieldother(\tau, \Cdot) \biggr\rangle_{\LpSpace^2(\domain,\Real^n)} \dint \tau \\
   &\quad\quad  +  \int_{0}^{1}\biggl\langle \frac{2\mu_2}{N}\sum_{i=1}^{N}\velocityfield_{\tau, t_i}(\Cdot),  \velocityfieldother(\tau, \Cdot)\biggr\rangle_{\Vspace}\dint \tau\\
  &=  \int_{0}^{1} \biggl\langle -\frac{2}{N}\sum_{\{i\geq 1 : t_i \geq t\}}h_{t, t_i}^{\template,\velocityfield} \nabla (\template \circ \gelement{t,0}{\velocityfield}),  \velocityfieldother(t, \Cdot) \biggr\rangle_{\LpSpace^2(\domain,\Real^n)} \dint t \\
   &\quad\quad +  \int_{0}^{1}\biggl\langle \frac{2\mu_2}{N}\sum_{\{i\geq 1 : t_i \geq t\}}\velocityfield_{t, t_i}(\Cdot),  \velocityfieldother(t, \Cdot)\biggr\rangle_{\Vspace}\dint t.
  \end{align*}
The last two equations are obtained by inserting \cref{eq:Middle_func_deriv_1} and \cref{eq:Middle_func_deriv_2}. 
Combining the above with \cref{eq:RKHS_L2} proves  \cref{eq:Energy_functional_time_discretized_gradient}.
\end{proof}

\section*{Acknowledgments}
The authors would like to thank Alain Trouv\'e for his helpful discussions.

\bibliographystyle{plain}
\bibliography{shapereferences}

\end{document}